\documentclass{amsart}%
\usepackage{amssymb}
\usepackage{amsmath}
\usepackage{amsfonts}%
\usepackage{graphicx}
\providecommand{\U}[1]{\protect\rule{.1in}{.1in}}
\newtheorem{theorem}{Theorem}
\theoremstyle{plain}

\newtheorem{corollary}{Corollary}

\newtheorem{definition}{Definition}
\newtheorem{example}{Example}

\newtheorem{lemma}{Lemma}

\newtheorem{proposition}{Proposition}
\newtheorem{remark}{Remark}

\numberwithin{equation}{section}
\begin{document}
\title[The Fractional Half Dirac Operator over Sobolev Spaces]{The Fractional Half Dirac Operator over Sobolev Spaces}
\author{Dejenie A. Lakew}
\address{Hampton University\\
Department of Mathematics}
\email{dejenie.lakew@hamptonu.edu}
\urladdr{http://www.hamptonu.edu}
\thanks{This paper is in final form and no version of it will be submitted for
publication elsewhere.}
\date{May $13,2026$}
\subjclass[$2000$]{Primary $46E35$, Secondary $46C15$}

\begin{abstract}
Let $\Omega$ be a bounded and smooth doamin in $%
%TCIMACRO{\U{211d} }%
%BeginExpansion
\mathbb{R}
%EndExpansion
^{n}$ and $D^{\frac{1}{2}}:=%
%TCIMACRO{\dsum \limits_{i=1}^{n}}%
%BeginExpansion
{\displaystyle\sum\limits_{i=1}^{n}}
%EndExpansion
e_{j}\frac{\partial^{\frac{1}{2}}}{\partial x^{\frac{1}{2}}}$ be the Dirac
half order differential operator. Let $f\in%
%TCIMACRO{\tciLaplace}%
%BeginExpansion
\mathcal{L}%
%EndExpansion
^{2}\left(  \Omega\right)  ,g\in%
%TCIMACRO{\tciLaplace}%
%BeginExpansion
\mathcal{L}%
%EndExpansion
^{2}\left(  \partial\Omega\right)  $. For $\alpha=1,2,...,$we consider BVPs
\[
\left\{
\begin{array}
[c]{cc}%
\left(  D^{\frac{1}{2}}\right)  ^{\alpha}u=f & \text{in }\Omega\\
\tau u=g & \text{on }\partial\Omega
\end{array}
\right.
\]
where $\tau$ is the trace operator.

The solution $u$ $\in$ $W^{\frac{1}{2},2}\left(  \Omega\right)  \ni
u=[u]_{g}\uplus\lbrack u]_{f}$ where $[u]_{g}:$ part of the solution that
evolves from the trace value $g$ and

$[u]_{f}:$ part that evolves from $f$ of the half order or of first order
solution over $\Omega$.

The symbol $\uplus$ represents an orthogonal sum of functions that are from
orthogonal sum $\oplus$ of subspaces of a Sobolev space with inner product.

\end{abstract}
\maketitle

\section{Introduction}

We introduce the fractional Dirac operator of a half order:

\begin{definition}%
\[
D^{\frac{1}{2}}:=%
%TCIMACRO{\dsum \limits_{i=1}^{n}}%
%BeginExpansion
{\displaystyle\sum\limits_{i=1}^{n}}
%EndExpansion
e_{j}\frac{\partial^{\frac{1}{2}}}{\partial x_{j}^{\frac{1}{2}}}:=%
%TCIMACRO{\dsum \limits_{i=1}^{n}}%
%BeginExpansion
{\displaystyle\sum\limits_{i=1}^{n}}
%EndExpansion
e_{j}\partial x_{j}^{\frac{1}{2}}%
\]
in a Clifford/hypercomplex algebra with basis vectors, $e_{1},e_{2}%
,e_{3},...,e_{n}$ and multiplicative properties:
\[
e_{i}^{2}=-1\text{ and }e_{i}e_{j}+e_{j}e_{i}=-\delta_{ij}%
\]
where $\delta$ is the Kronecker delta function.
\end{definition}

We develop some functional theoretical results based on this fractional Dirac
operator and see the application of the theory in an applied fields.

Note that
\[
\left(  D^{\frac{1}{2}}\right)  ^{2}=%
%TCIMACRO{\dsum \limits_{j=1}^{n}}%
%BeginExpansion
{\displaystyle\sum\limits_{j=1}^{n}}
%EndExpansion
e_{j}^{2}\frac{\partial}{\partial x_{j}}+%
%TCIMACRO{\dsum \limits_{j<k}^{n}}%
%BeginExpansion
{\displaystyle\sum\limits_{j<k}^{n}}
%EndExpansion
e_{j}^{2}\{e_{j},e_{k},\}=-%
%TCIMACRO{\dsum \limits_{k=1}^{n}}%
%BeginExpansion
{\displaystyle\sum\limits_{k=1}^{n}}
%EndExpansion
\frac{\partial}{\partial x_{k}}\neq D
\]

where $D$ is the regular vector Dirac operator given below, but $\left(
D^{\frac{1}{2}}\right)  ^{2}$ is a scalar partial differential operator.

\begin{definition}
The Dirac operator is defined by
\[
D=%
%TCIMACRO{\dsum \limits_{i=1}^{n}}%
%BeginExpansion
{\displaystyle\sum\limits_{i=1}^{n}}
%EndExpansion
e_{j}\frac{\partial}{\partial x_{j}}%
\]
a partial differential operator of first order in a hyper complex or Clifford
analysis setting \ 
\end{definition}

There is a good amount of work done regarding the Dirac operator. Our work
focuses on the fractional half order Dirac derivative $D$. \ To have some
observation on fractional order derivatives, the regular ordinary derivative
$\frac{d}{dx}f(x)$ of a differentiable function $f$ makes the function to
loose smoothness or regularity by $1$, while the half order derivative makes
the function to loose $\frac{1}{2}-$regularity exponent. Lose of regularity
creates discontiniuty on functions that are rough on their graphs.

The lose of half order rate of change different from the regular lose of $1$
in differentiation, has physical meanings in applied sciences. For instance,
if a particle moves in a media of resistance, then the half order rate of
change of the position displays the reduced speed in the media due to the
resistance the object encounters, such as viscousity,wave fronts or others.

Let us consider the half order derivative $D^{\frac{1}{2}}=\frac{d^{\frac
{1}{2}}}{dx^{\frac{1}{2}}}$of $x^{\lambda}$ Although there are few more
fractional derivatives, we consider what are called the Caputo and
Riemann-Lioville fractional derivatives.

The Caputo fractional half derivative of a function $f$ defined on an interval
$(0,1)\subseteq%
%TCIMACRO{\U{211d} }%
%BeginExpansion
\mathbb{R}
%EndExpansion
^{1}$ at $x\in\left(  0,1\right)  $%
\[
\underset{0}{C}D_{x}^{\frac{1}{2}}f(x)=\frac{1}{\Gamma\left(  \frac{1}%
{2}\right)  }%
%TCIMACRO{\dint \limits_{0}^{x}}%
%BeginExpansion
{\displaystyle\int\limits_{0}^{x}}
%EndExpansion
\frac{f^{\prime}\left(  t\right)  }{\left(  x-t\right)  ^{\frac{1}{2}}}dt
\]
Clearly when $f(x)=c-$ constant, the \textit{Caputo} fractional half
derivative $D_{x}^{\frac{1}{2}}f(x)=0$.

When $f(x)=x^{\lambda}$, we have:
\begin{align*}
\underset{0}{C}D_{x}^{\frac{1}{2}}(x^{\lambda})  & =\frac{1}{\Gamma\left(
\frac{1}{2}\right)  }%
%TCIMACRO{\dint \limits_{0}^{x}}%
%BeginExpansion
{\displaystyle\int\limits_{0}^{x}}
%EndExpansion
\frac{\lambda t^{\lambda-1}}{\left(  x-t\right)  ^{\frac{1}{2}}}dt\\
& =\frac{\Gamma\left(  \lambda+1\right)  }{\Gamma\left(  \lambda+\frac{1}%
{2}\right)  }x^{\lambda-\frac{1}{2}}%
\end{align*}
When $\alpha=\frac{1}{2}$ and $\lambda=\frac{1}{2}$we have
\[
\underset{0}{C}D_{x}^{\frac{1}{2}}(x^{\frac{1}{2}})=\frac{\sqrt{\pi}}{2}%
\]
One can observe that the Caputo fractional derivative and the regular
derivative of a constant coincides.

The $\alpha^{th}-$order \textit{Riemann-Lioville} fractional derivative: for
$0\leq\alpha<1,$%

\[
^{RL}D_{x}^{\alpha}f(x)=\frac{d}{dx}\int_{0}^{x}\frac{f(\tau)}{\left(
x-\tau\right)  ^{\alpha}}d\tau
\]

Note that the constant function $f(x)=c$ behaves differently in this case

\textit{Rieman-Lioville} $\frac{1}{2}^{th}-$derivative:
\[
^{RL}D_{x}^{\frac{1}{2}}(c)=\frac{d}{dx}\int_{0}^{x}\frac{c}{\left(
x-\tau\right)  ^{\alpha}}d\tau=c\frac{x^{-\frac{1}{2}}}{\Gamma\left(  \frac
{1}{2}\right)  }=c\frac{1}{\sqrt{\pi x}}%
\]

Note here that
\begin{align*}
^{RL}D_{x}^{\frac{1}{2}}(^{RL}D_{x}^{\frac{1}{2}}(c))  & =^{RL}D_{x}^{\frac
{1}{2}}(c\frac{1}{\sqrt{\pi x}})\\
& =\frac{c}{\pi}\frac{d}{dx}%
%TCIMACRO{\dint \limits_{0}^{x}}%
%BeginExpansion
{\displaystyle\int\limits_{0}^{x}}
%EndExpansion
\frac{1}{\sqrt{tx-t^{2}}}dt,\because%
%TCIMACRO{\dint \limits_{0}^{x}}%
%BeginExpansion
{\displaystyle\int\limits_{0}^{x}}
%EndExpansion
\frac{1}{\sqrt{tx-t^{2}}}dt=\pi\\
& =0
\end{align*}
which signifies:
\[
^{RL}D_{x}^{\frac{1}{2}}(^{RL}D_{x}^{\frac{1}{2}}\left(  f(x)\right)
=^{RL}D_{x}f(x)
\]
which is the regular first order derivative of $f(x),$i.e., $^{RL}%
D_{x}f(x)=f^{\prime}(x)$ to be shown next.

\begin{corollary}
For $\alpha=1,$ the Reimann-Lioville first order derivative of a function $f $
is the regular first order derivative. That is $^{RL}D^{1}f(x)=\frac{d}%
{dx}f(x)$.
\end{corollary}

\begin{proof}
For $\alpha\in\left(  0,1\right)  :$
\[
^{RL}D_{x}^{\alpha}f(x)=\frac{1}{\Gamma\left(  1-\alpha\right)  }\frac{d}{dx}%
%TCIMACRO{\dint \limits_{0}^{x}}%
%BeginExpansion
{\displaystyle\int\limits_{0}^{x}}
%EndExpansion
\frac{f(t)}{\left(  x-t\right)  ^{\alpha}}dt
\]
Then
\[
^{RL}D_{x}^{1}f(x):=\underset{\alpha\longrightarrow1^{-}}{\lim}\frac{1}%
{\Gamma\left(  1-\alpha\right)  }\frac{d}{dx}%
%TCIMACRO{\dint \limits_{0}^{x}}%
%BeginExpansion
{\displaystyle\int\limits_{0}^{x}}
%EndExpansion
\frac{f(t)}{\left(  x-t\right)  ^{\alpha}}dt=\frac{d}{dx}\underset{\alpha
\longrightarrow1^{-}}{\lim}\frac{1}{\Gamma\left(  1-\alpha\right)  }%
%TCIMACRO{\dint \limits_{0}^{x}}%
%BeginExpansion
{\displaystyle\int\limits_{0}^{x}}
%EndExpansion
\frac{f(t)}{\left(  x-t\right)  ^{\alpha}}dt
\]

But
\[%
%TCIMACRO{\dint \limits_{0}^{x}}%
%BeginExpansion
{\displaystyle\int\limits_{0}^{x}}
%EndExpansion
\frac{f(t)}{\left(  x-t\right)  ^{\alpha}}dt=\frac{f(0)x^{1-\alpha}}{1-\alpha
}+\frac{1}{1-\alpha}%
%TCIMACRO{\dint \limits_{0}^{x}}%
%BeginExpansion
{\displaystyle\int\limits_{0}^{x}}
%EndExpansion
\frac{f^{\prime}(t)}{(x-t)^{\alpha-1}}%
\]

Therefore,
\begin{align*}
\underset{\alpha\longrightarrow1^{-}}{\lim}\frac{1}{\Gamma\left(
1-\alpha\right)  }\left(  \frac{f(0)x^{1-\alpha}}{1-\alpha}+\frac{1}{1-\alpha}%
%TCIMACRO{\dint \limits_{0}^{x}}%
%BeginExpansion
{\displaystyle\int\limits_{0}^{x}}
%EndExpansion
\frac{f^{\prime}(t)}{(x-t)^{\alpha-1}}\right)   & =\underset{\alpha
\longrightarrow1^{-}}{\lim}\frac{1}{(1-\alpha)\Gamma\left(  1-\alpha\right)
}\left(  f(0)x^{1-\alpha}+%
%TCIMACRO{\dint \limits_{0}^{x}}%
%BeginExpansion
{\displaystyle\int\limits_{0}^{x}}
%EndExpansion
\frac{f^{\prime}(t)}{(x-t)^{\alpha-1}}\right) \\
& =\underset{\alpha\longrightarrow1^{-}}{\lim}\frac{1}{\Gamma\left(
2-\alpha\right)  }\left(  f(0)x^{1-\alpha}+%
%TCIMACRO{\dint \limits_{0}^{x}}%
%BeginExpansion
{\displaystyle\int\limits_{0}^{x}}
%EndExpansion
\frac{f^{\prime}(t)}{(x-t)^{\alpha-1}}\right) \\
& =f(0)+%
%TCIMACRO{\dint \limits_{0}^{x}}%
%BeginExpansion
{\displaystyle\int\limits_{0}^{x}}
%EndExpansion
f^{\prime}(t)dt=f(x),\text{ \ \ }\because\Gamma\left(  1\right)  =1
\end{align*}

Hence,
\[
^{RL}D_{x}^{1}f(x)=\frac{d}{dx}f(x)
\]

\end{proof}

Clearly, both derivatives, the Caputo fractional derivative and the
Riemann-Lioville derivative of a constant offers two contrasting results. One,
constants have invariant rate of change, which is zero and the other one says,
it behaves like a non constant if the object is in resistive media such as viscosity.

In this work, we study orthogonally decomposable solutions to boundary value
problems of a fractional half or twice of a half orders in an inner product
Sobolev space. We establish that solutions to BVPs to be orthogonal sums of
functions that evolve from the trace values on the boundary and the values of
the differential equation in the domain. In \cite{dlakew1},\cite{dlakew2}%
,\cite{lakewryan1},\cite{lakewryan2} we successfully developed regular
decomposition results of Hilbert and Sobolev spaces with an inner product. We
established properties of the inner product that are considered and that of
functions in the respective spaces. In \cite{dlakew1} we have seen how norm is
enlarged and space is expanding when the regularity exponent increases.

\begin{definition}
Let $\Omega$ be a bounded domain in $%
%TCIMACRO{\U{211d} }%
%BeginExpansion
\mathbb{R}
%EndExpansion
^{n}$ with a boundary $\partial\Omega$. We define the fractional Sobolev space
of regularity $\frac{1}{2},$ denoted by $W^{\frac{1}{2},2}\left(
\Omega\right)  $ as $W^{\frac{1}{2},2}\left(  \Omega\right)  :=\{f\in%
%TCIMACRO{\tciLaplace}%
%BeginExpansion
\mathcal{L}%
%EndExpansion
^{2}\left(  \Omega\right)  :D^{\frac{1}{2}}f\in%
%TCIMACRO{\tciLaplace}%
%BeginExpansion
\mathcal{L}%
%EndExpansion
^{2}\left(  \Omega\right)  \}$ where $D^{\frac{1}{2}}f$ is the fractional
$\frac{1}{2}-$Caputo derivative of $f$.
\end{definition}

\begin{proposition}
For $f,g\in%
%TCIMACRO{\tciLaplace}%
%BeginExpansion
\mathcal{L}%
%EndExpansion
^{2}\left(  \Omega\right)  ,$ the BVP:
\begin{equation}
\left\{
\begin{array}
[c]{cc}%
D^{\frac{1}{2}}u=f & \text{in }\Omega\\
u=g & \text{on }\partial\Omega
\end{array}
\right. \label{bvp1}%
\end{equation}
\ has a solution $u$ in the Sobolev space $W^{\frac{1}{2},2}(\Omega)$ of an
inner product $\left\langle .,.\right\rangle _{W^{\frac{1}{2},2}\left(
\Omega\right)  }$ such that $u$ is an orthogonal sum of its parts, i.e.,
\begin{equation}
u=\left[  u\right]  _{f}\uplus\left[  u\right]  _{g}\label{dirctsum1}%
\end{equation}

\end{proposition}

with $\left[  u\right]  _{f}$ be the part of the solution that evolves from
$f$ and $\left[  u\right]  _{g}$ the part that evolves from $g$ with properties:%

\[
\left(  i\right)  \ \ \ \ \ \ \ \ \ \left\langle \left[  u\right]
_{f},\left[  u\right]  _{g}\right\rangle _{W^{\frac{1}{2},2}\left(
\Omega\right)  }=0
\]

\ \ %

\begin{align*}
\left(  ii\right)  \ \ \ \ \ \ \ \ \ \Vert u\Vert_{W^{\frac{1}{2},2}\left(
\Omega\right)  }^{2}  & =\Vert\left[  u\right]  _{f}\Vert_{W^{\frac{1}{2}%
,2}\left(  \Omega\right)  }^{2}+\Vert\left[  u\right]  _{g}\Vert_{W^{\frac
{1}{2},2}\left(  \Omega\right)  }^{2}\ \ \ \ \ \ \\
&
\end{align*}

\begin{proof}
From integration by parts:%
\[
\underset{\underset{u(x)}{\parallel}}{\underbrace{\int_{\Omega}u(y)D_{y}%
^{\frac{1}{2}}\phi_{\frac{1}{2}}\left(  y-x\right)  dy}}%
=\underset{\overset{\parallel}{\underset{\zeta_{\partial\Omega}^{\frac{1}{2}%
}(g)}{\underset{\parallel}{\zeta_{\partial\Omega}^{\frac{1}{2}}\tau\left(
u\right)  }}}}{\underbrace{\int_{\partial\Omega}u(y)\phi_{\frac{1}{2}}\left(
y-x\right)  \nu_{\frac{1}{2}}\left(  y\right)  d\partial\Omega_{y}}%
}\underset{\underset{\zeta_{\Omega}^{\frac{1}{2}}\left(  f\right)
}{\parallel}}{\underbrace{-\int_{\Omega}D^{\frac{1}{2}}u(y)\phi_{\frac{1}{2}%
}\left(  y-x\right)  d\Omega_{y}}}%
\]

That is%
\begin{equation}
u(x)=\zeta_{\partial\Omega}^{\frac{1}{2}}(g)+\zeta_{\Omega}^{\frac{1}{2}%
}\left(  f\right) \label{sol1}%
\end{equation}

From the orthogonal decomposition
\[
W^{\frac{1}{2},2}\left(  \Omega\right)  =KerD^{\frac{1}{2}}\cap W^{\frac{1}%
{2},2}\left(  \Omega\right)  \oplus D^{\frac{1}{2}}\left(  W_{0}^{1,2}\left(
\Omega\right)  \right)
\]

taking the Dirac half derivative $D^{\frac{1}{2}}$ of both sides of the above
integral equation, we have $D^{\frac{1}{2}}\zeta_{\partial\Omega}^{\frac{1}%
{2}}(g)=0$
\end{proof}

\begin{example}
Consider $f(x)=\left\{
\begin{array}
[c]{cc}%
x-1, & 1\leq x\leq2\\
0, & 0\leq x\leq1
\end{array}
\right.  $

Then
\[
g(x)=\left\{
\begin{array}
[c]{cc}%
\frac{2\sqrt{x-1}}{\sqrt{\pi}}, & 1<x<2\\
0, & 0<x<1
\end{array}
\right.
\]

\end{example}

is the $\frac{1}{2}-$derivative of $f$ in the Riemann-Lioville sense

Indeed,
\begin{align*}
^{RL}D^{\frac{1}{2}}f(x)  & =\frac{1}{\sqrt{\pi}}\frac{d}{dx}\int_{0}^{x}%
\frac{f(t)}{\left(  x-t\right)  ^{\frac{1}{2}}}dt\\
& =\frac{1}{\sqrt{\pi}}\frac{d}{dx}\int_{1}^{x}\frac{t-1}{\left(  x-t\right)
^{\frac{1}{2}}}dt\\
& =\frac{1}{\sqrt{\pi}}2\sqrt{x-1}.
\end{align*}

Therefore,
\begin{align*}
^{RL}D^{\frac{1}{2}}f(x)  & =\left\{
\begin{array}
[c]{cc}%
\frac{2\sqrt{x-1}}{\sqrt{\pi}}, & 1<x<2\\
0, & 0<x<1
\end{array}
\right. \\
& =g(x)
\end{align*}

Also $g$ is the weekly $\frac{1}{2}-$derivative of $f$ as well. That is
$\forall\phi\in C_{0}^{\infty}\left(  \left[  0,2\right]  \right)  ,$%

\[%
%TCIMACRO{\dint \limits_{0}^{2}}%
%BeginExpansion
{\displaystyle\int\limits_{0}^{2}}
%EndExpansion
g(x)\phi\left(  x\right)  dx=%
%TCIMACRO{\dint \limits_{0}^{2}}%
%BeginExpansion
{\displaystyle\int\limits_{0}^{2}}
%EndExpansion
f(x)_{x}D_{2}^{\frac{1}{2}}\phi\left(  x\right)  )dx.
\]

Now
\[
_{x}^{RL}D_{2}^{\frac{1}{2}}\phi\left(  x\right)  =-\frac{1}{\sqrt{\pi}}%
\frac{d}{dx}%
%TCIMACRO{\dint \limits_{x}^{2}}%
%BeginExpansion
{\displaystyle\int\limits_{x}^{2}}
%EndExpansion
\frac{\phi\left(  t\right)  }{\sqrt{t-x}}dt
\]
yields,
\begin{align*}%
%TCIMACRO{\dint \limits_{0}^{2}}%
%BeginExpansion
{\displaystyle\int\limits_{0}^{2}}
%EndExpansion
g(x)\phi\left(  x\right)  dx  & =%
%TCIMACRO{\dint \limits_{1}^{2}}%
%BeginExpansion
{\displaystyle\int\limits_{1}^{2}}
%EndExpansion
\frac{2\sqrt{x-1}}{\sqrt{\pi}}\phi\left(  x\right)  dx\\
& =%
%TCIMACRO{\dint \limits_{1}^{2}}%
%BeginExpansion
{\displaystyle\int\limits_{1}^{2}}
%EndExpansion
f(x)\left(  _{x}^{RL}D_{2}^{\frac{1}{2}}\phi\left(  x\right)  \right)  dx
\end{align*}

Therefore, $g(x)$ is the week $\frac{1}{2}-$derivative of $f.$

If we continue taking the Riemann-Lioville $\frac{1}{2}-$derivative of $g,$%
\begin{align*}
^{RL}D^{\frac{1}{2}}g\left(  x\right)   & =\frac{1}{\sqrt{\pi}}\frac{d}{dx}%
%TCIMACRO{\dint \limits_{0}^{x}}%
%BeginExpansion
{\displaystyle\int\limits_{0}^{x}}
%EndExpansion
\frac{2\sqrt{t-1}}{\sqrt{\pi}\sqrt{x-t}}dt\\
& =\frac{2}{\pi}\frac{d}{dx}\left(  \frac{\pi}{2}\left(  x-1\right)  \right)
\\
& =1.
\end{align*}

That is
\[
^{RL}D^{\frac{1}{2}}\left(  ^{RL}D^{\frac{1}{2}}(f(x))\right)  )=1=^{RL}%
D^{1}f(x)
\]
which is to be shown the regular ordinary derivative of first order.

Note here that $f$ is not smooth and hence not differentiable on the given
interval in a regular sense and $^{RL}D^{\frac{1}{2}}(1)\neq0$. Therefore the
only function that has a zero derivative in the Riemann-Lioville sense is the
zero function.

We introduce the fractional half Dirac operator in the Caputo sense.

\begin{definition}
The Dirac operator of order $\frac{1}{2}$ (Caputo sense) is defined to be
\[
D^{\frac{1}{2}}:=%
%TCIMACRO{\dsum \limits_{j=1}^{n}}%
%BeginExpansion
{\displaystyle\sum\limits_{j=1}^{n}}
%EndExpansion
e_{j}\frac{\partial^{\frac{1}{2}}}{\partial x_{j}^{\frac{1}{2}}}=:%
%TCIMACRO{\dsum \limits_{j=1}^{n}}%
%BeginExpansion
{\displaystyle\sum\limits_{j=1}^{n}}
%EndExpansion
e_{j}\partial_{x_{j}}^{\frac{1}{2}}%
\]

The fundamental solution to $D^{\frac{1}{2}}$ is given by
\[
\psi_{\frac{1}{2}}\left(  y-x\right)  =\frac{y-x}{\omega_{n}\parallel
y-x\parallel^{n+\frac{1}{2}}}%
\]
That is
\[
D^{\frac{1}{2}}\psi_{\frac{1}{2}}\left(  y-x\right)  =\delta\left(
y-x\right)
\]

where $\delta$ is the Kronecker delta function.
\end{definition}

\begin{example}
Let $f(x)=x=%
%TCIMACRO{\dsum \limits_{i=1}^{n}}%
%BeginExpansion
{\displaystyle\sum\limits_{i=1}^{n}}
%EndExpansion
e_{i}x_{i}$. Then the fractional $\frac{1}{2}-$Caputo derivative of $f$ is
\begin{align*}
D^{\frac{1}{2}}f(x)  & =%
%TCIMACRO{\dsum \limits_{i=1}^{n}}%
%BeginExpansion
{\displaystyle\sum\limits_{i=1}^{n}}
%EndExpansion
e_{i}\left(  \frac{1}{\sqrt{\pi}}%
%TCIMACRO{\dint \limits_{0}^{x_{i}}}%
%BeginExpansion
{\displaystyle\int\limits_{0}^{x_{i}}}
%EndExpansion
\frac{\partial f_{i}\left(  x_{1},x_{2},...,t_{i},...,x_{n}\right)  }%
{\sqrt{x_{i}-t_{t}}}dt_{i}\right) \\
& =\frac{2}{\sqrt{\pi}}%
%TCIMACRO{\dsum \limits_{i=1}^{n}}%
%BeginExpansion
{\displaystyle\sum\limits_{i=1}^{n}}
%EndExpansion
e_{i}\sqrt{x_{i}}%
\end{align*}

\end{example}

\begin{example}
What happens, when we differentiate $D^{\frac{1}{2}}f(x)$ with half order in
the Caputo sense again?%

\begin{align*}
D^{\frac{1}{2}}(D^{\frac{1}{2}}f(x))  & =D^{\frac{1}{2}}\left(
%TCIMACRO{\dsum \limits_{i=1}^{n}}%
%BeginExpansion
{\displaystyle\sum\limits_{i=1}^{n}}
%EndExpansion
e_{i}\frac{2}{\sqrt{\pi}}\sqrt{x_{i}}\right) \\
& =%
%TCIMACRO{\dsum \limits_{i=1}^{n}}%
%BeginExpansion
{\displaystyle\sum\limits_{i=1}^{n}}
%EndExpansion
e_{i}D_{x_{i}}^{\frac{1}{2}}\left(  2\sqrt{\frac{x_{i}}{\pi}}\right) \\
& =%
%TCIMACRO{\dsum \limits_{i=1}^{n}}%
%BeginExpansion
{\displaystyle\sum\limits_{i=1}^{n}}
%EndExpansion
e_{i}=\mathbf{1}%
\end{align*}

\end{example}

Note here that
\[
\mathbf{1=}%
%TCIMACRO{\dsum \limits_{i=1}^{n}}%
%BeginExpansion
{\displaystyle\sum\limits_{i=1}^{n}}
%EndExpansion
e_{i}=D\left(  \mathbf{x}\right)
\]
with $\mathbf{x}=%
%TCIMACRO{\dsum \limits_{i=1}^{n}}%
%BeginExpansion
{\displaystyle\sum\limits_{i=1}^{n}}
%EndExpansion
e_{i}x_{i}$, the regular vector Dirac derivative of first order. But again,%
\[
D^{\frac{1}{2}}\left(  \mathbf{1}\right)  =0
\]

Therefore, we see
\[
D^{\frac{1}{2}}\left(  D^{\frac{1}{2}}\left(  D^{\frac{1}{2}}\left(
\mathbf{x}\right)  \right)  \right)  =0=D^{1+\frac{1}{2}}\left(
\mathbf{x}\right)  =%
%TCIMACRO{\dsum \limits_{j=1}^{n}}%
%BeginExpansion
{\displaystyle\sum\limits_{j=1}^{n}}
%EndExpansion
e_{j}\frac{\partial^{1+\frac{1}{2}}}{\partial x_{j}^{1+\frac{1}{2}}}\left(
\mathbf{x}\right)  =%
%TCIMACRO{\dsum \limits_{j=1}^{n}}%
%BeginExpansion
{\displaystyle\sum\limits_{j=1}^{n}}
%EndExpansion
e_{j}\frac{\partial^{\frac{1}{2}}}{\partial x_{j}^{\frac{1}{2}}}\left(
\mathbf{1}\right)  =D^{\frac{1}{2}}\left(  \mathbf{1}\right)
\]

We introduce two integral fractional transforms $\zeta_{\Omega}^{\frac{1}{2}}$
(Teodorescu) and $\zeta_{\partial\Omega}^{\frac{1}{2}}$ (Feuter)$,$which are
domain and boundary integrals respectively.

\begin{definition}
For $f\in%
%TCIMACRO{\tciLaplace}%
%BeginExpansion
\mathcal{L}%
%EndExpansion
^{2}\left(  \Omega\right)  $%
\[
\zeta_{\Omega}^{\frac{1}{2}}f(x)=%
%TCIMACRO{\dint \limits_{\Omega}}%
%BeginExpansion
{\displaystyle\int\limits_{\Omega}}
%EndExpansion
\psi_{\frac{1}{2}}(y-x)f(y)d\Omega_{y}:\text{ domain fractional integral}%
\]
and
\[
\zeta_{\partial\Omega}^{\frac{1}{2}}f(x)=%
%TCIMACRO{\dint \limits_{\partial\Omega}}%
%BeginExpansion
{\displaystyle\int\limits_{\partial\Omega}}
%EndExpansion
\psi_{\frac{1}{2}}(y-x)v_{\frac{1}{2}}\left(  y\right)  f(y)d\partial
\Omega_{y}:\text{ boundary fractional integral}%
\]
where $\upsilon_{\frac{1}{2}}\left(  y\right)  $ is the outward fractional
half unit normal vector at $y\in\partial\Omega$ and $\psi_{\frac{1}{2}}(y-x)$
is the fundamental solution to the Caputo $D^{\frac{1}{2}}$.
\end{definition}

One can show that $\zeta_{\Omega}^{\frac{1}{2}}$ is the right algebraic
inverse of the half Dirac operator $D^{\frac{1}{2}}$ and $\zeta_{\partial
\Omega}^{\frac{1}{2}}$ is in its kernel. That is%

\[
D^{\frac{1}{2}}\left(  \zeta_{\Omega}^{\frac{1}{2}}\left(  f\right)  \right)
=f\text{ \ \ and \ \ }D^{\frac{1}{2}}\left(  \zeta_{\partial\Omega}^{\frac
{1}{2}}\left(  f\right)  \right)  =0
\]

Let $f\in W^{\frac{1}{2},2}\left(  \Omega\right)  $, and $\psi_{\frac{1}{2}%
}\left(  y-x\right)  $ be the fundamental solution to the Caputo fractional
half Dirac operator $D^{\frac{1}{2}}.$ Then the method of integration by parts
yields: \-%
\begin{align*}%
%TCIMACRO{\dint \limits_{\Omega}}%
%BeginExpansion
{\displaystyle\int\limits_{\Omega}}
%EndExpansion
\psi_{\frac{1}{2}}\left(  y-x\right)  D^{\frac{1}{2}}f(y)d\Omega_{y}  & =%
%TCIMACRO{\dint \limits_{\partial\Omega}}%
%BeginExpansion
{\displaystyle\int\limits_{\partial\Omega}}
%EndExpansion
\psi_{\frac{1}{2}}\left(  y-x\right)  \upsilon_{\frac{1}{2}}\left(  y\right)
f(y)d\partial\Omega_{y}-%
%TCIMACRO{\dint \limits_{\Omega}}%
%BeginExpansion
{\displaystyle\int\limits_{\Omega}}
%EndExpansion
D^{\frac{1}{2}}\psi_{\frac{1}{2}}\left(  y-x\right)  f(y)d\Omega_{y}\\
& =%
%TCIMACRO{\dint \limits_{\partial\Omega}}%
%BeginExpansion
{\displaystyle\int\limits_{\partial\Omega}}
%EndExpansion
\psi_{\frac{1}{2}}\left(  y-x\right)  \upsilon_{\frac{1}{2}}\left(  y\right)
f(y)d\partial\Omega_{y}-%
%TCIMACRO{\dint \limits_{\Omega}}%
%BeginExpansion
{\displaystyle\int\limits_{\Omega}}
%EndExpansion
\delta\left(  y-x\right)  f(y)d\Omega_{y}\\
& =%
%TCIMACRO{\dint \limits_{\partial\Omega}}%
%BeginExpansion
{\displaystyle\int\limits_{\partial\Omega}}
%EndExpansion
\psi_{\frac{1}{2}}\left(  y-x\right)  \upsilon_{\frac{1}{2}}\left(  y\right)
f(y)d\partial\Omega_{y}-f(x)
\end{align*}

Therefore we can write the following:

\begin{proposition}
For $f\in W^{\frac{1}{2},2}\left(  \Omega\right)  ,$%
\[
f(x)=%
%TCIMACRO{\dint \limits_{\partial\Omega}}%
%BeginExpansion
{\displaystyle\int\limits_{\partial\Omega}}
%EndExpansion
\psi_{\frac{1}{2}}\left(  y-x\right)  \upsilon_{\frac{1}{2}}\left(  y\right)
f(y)d\partial\Omega_{y}-%
%TCIMACRO{\dint \limits_{\Omega}}%
%BeginExpansion
{\displaystyle\int\limits_{\Omega}}
%EndExpansion
\psi_{\frac{1}{2}}\left(  y-x\right)  D^{\frac{1}{2}}f(y)d\Omega_{y}%
\]

\end{proposition}

\begin{proposition}
Let $f,g\in%
%TCIMACRO{\tciLaplace}%
%BeginExpansion
\mathcal{L}%
%EndExpansion
^{2}\left(  \Omega\right)  ,$ then the BVP:
\[
\left\{
\begin{array}
[c]{cc}%
D^{\frac{1}{2}}u=f, & \text{in }\Omega\\
tru=g, & \text{on }\partial\Omega
\end{array}
\right.
\]

has a solution given by
\begin{align*}
u(x)  & =\zeta_{\partial\Omega}^{\frac{1}{2}}g(x)-\zeta_{\Omega}^{\frac{1}{2}%
}f(x)\\
& =%
%TCIMACRO{\dint \limits_{\partial\Omega}}%
%BeginExpansion
{\displaystyle\int\limits_{\partial\Omega}}
%EndExpansion
\psi_{\frac{1}{2}}\left(  y-x\right)  \upsilon_{\frac{1}{2}}\left(  y\right)
g(y)d\partial\Omega_{y}-%
%TCIMACRO{\dint \limits_{\Omega}}%
%BeginExpansion
{\displaystyle\int\limits_{\Omega}}
%EndExpansion
\psi_{\frac{1}{2}}\left(  y-x\right)  f(y)d\Omega_{y}%
\end{align*}

\end{proposition}

\begin{proof}
The proof follows from the argument of integration by parts and taking the
trace value of $u$ on the boundary $\partial\Omega$ and the fractional
$\frac{1}{2}$derivative of $u$ in the domain $\Omega$.
\end{proof}

\begin{lemma}
For $g\in%
%TCIMACRO{\tciLaplace}%
%BeginExpansion
\mathcal{L}%
%EndExpansion
^{2}\left(  \Omega\right)  $, the function $\phi\left(  x\right)  =%
%TCIMACRO{\dint \limits_{\partial\Omega}}%
%BeginExpansion
{\displaystyle\int\limits_{\partial\Omega}}
%EndExpansion
\psi_{\frac{1}{2}}\left(  y-x\right)  \upsilon_{\frac{1}{2}}\left(  y\right)
g(y)d\partial\Omega_{y}$ is $\frac{1}{2}-$monogenic or $\frac{1}{2}-$analytic
over $\Omega$.
\end{lemma}

\begin{proof}
We need to show that for $x\in\Omega\setminus\partial\Omega:$
\[
D^{\frac{1}{2}}\phi\left(  x\right)  =D^{\frac{1}{2}}%
%TCIMACRO{\dint \limits_{\partial\Omega}}%
%BeginExpansion
{\displaystyle\int\limits_{\partial\Omega}}
%EndExpansion
\psi_{\frac{1}{2}}\left(  y-x\right)  \upsilon_{\frac{1}{2}}\left(  y\right)
g(y)d\partial\Omega_{y}=%
%TCIMACRO{\dint \limits_{\partial\Omega}}%
%BeginExpansion
{\displaystyle\int\limits_{\partial\Omega}}
%EndExpansion
D_{x}^{\frac{1}{2}}\left(  \psi_{\frac{1}{2}}\left(  y-x\right)
\upsilon_{\frac{1}{2}}\left(  y\right)  g(y)\right)  d\partial\Omega_{y}=0
\]

Differently, we have to show that%
\[
D_{x}^{\frac{1}{2}}\left(  \psi_{\frac{1}{2}}\left(  y-x\right)  \right)
=0\text{ \ in the above integral. }%
\]
But $\psi_{\frac{1}{2}}\left(  x-y\right)  $ is the fundamnetal solution to
the fractional $\frac{1}{2}$-Dirac operator $D_{x}^{\frac{1}{2}}$.

Therefore,
\[
D_{x}^{\frac{1}{2}}\left(  \psi_{\frac{1}{2}}\left(  y-x\right)  \right)
=\delta\left(  y-x\right)  =\left\{
\begin{array}
[c]{cc}%
0, & x\neq y\\
\infty, & x=y
\end{array}
\right.
\]

But $\phi$ is defined at points $x\in\Omega\setminus\partial\Omega$ and point
$y\in\partial\Omega$. Therefore $\parallel y-x\parallel>0$ and hence $x\neq y$
which implies $\delta\left(  y-x\right)  =0$.%

\[
\therefore\text{ \ \ \ \ \ \ \ \ \ \ \ }%
%TCIMACRO{\dint \limits_{\partial\Omega}}%
%BeginExpansion
{\displaystyle\int\limits_{\partial\Omega}}
%EndExpansion
D_{x}^{\frac{1}{2}}\left(  \psi_{\frac{1}{2}}\left(  y-x\right)
\upsilon_{\frac{1}{2}}\left(  y\right)  g(y)\right)  d\partial\Omega_{y}=0
\]

\[
\Longrightarrow\text{ \ \ \ \ \ \ \ \ }D^{\frac{1}{2}}%
%TCIMACRO{\dint \limits_{\partial\Omega}}%
%BeginExpansion
{\displaystyle\int\limits_{\partial\Omega}}
%EndExpansion
\psi_{\frac{1}{2}}\left(  y-x\right)  \upsilon_{\frac{1}{2}}\left(  y\right)
g(y)d\partial\Omega_{y}=%
%TCIMACRO{\dint \limits_{\partial\Omega}}%
%BeginExpansion
{\displaystyle\int\limits_{\partial\Omega}}
%EndExpansion
D_{x}^{\frac{1}{2}}\left(  \psi_{\frac{1}{2}}\left(  y-x\right)
\upsilon_{\frac{1}{2}}\left(  y\right)  g(y)\right)  d\partial\Omega_{y}=0
\]

Hence,%
\[
\phi\left(  x\right)  =%
%TCIMACRO{\dint \limits_{\partial\Omega}}%
%BeginExpansion
{\displaystyle\int\limits_{\partial\Omega}}
%EndExpansion
\psi_{\frac{1}{2}}\left(  y-x\right)  \upsilon_{\frac{1}{2}}\left(  y\right)
g(y)d\partial\Omega_{y}\in KerD^{\frac{1}{2}}\left(  \Omega\right)
\]

\end{proof}

\begin{corollary}
For a function $u$ whose trace $u_{\mid\partial\Omega}=0$ and $%
\begin{array}
[c]{cc}%
D^{\frac{1}{2}}u=f, & \text{in }\Omega
\end{array}
$, the solution is given by
\[
u\left(  x\right)  =-%
%TCIMACRO{\dint \limits_{\Omega}}%
%BeginExpansion
{\displaystyle\int\limits_{\Omega}}
%EndExpansion
\psi_{\frac{1}{2}}\left(  y-x\right)  f(y)d\Omega_{y}%
\]
Likewise a function $u$ which is $\frac{1}{2}-$monogenic on $\Omega$ and trace
value $g$ on $\partial\Omega$ is given by
\[
u(x)=%
%TCIMACRO{\dint \limits_{\partial\Omega}}%
%BeginExpansion
{\displaystyle\int\limits_{\partial\Omega}}
%EndExpansion
\psi_{\frac{1}{2}}\left(  y-x\right)  \upsilon_{\frac{1}{2}}\left(  y\right)
g(y)d\partial\Omega_{y}%
\]

\end{corollary}

\begin{remark}
Note that the solution $u$ for both cases is in the fractional Sobolev space
$W^{\frac{1}{2},2}\left(  \Omega\right)  $.
\end{remark}

\begin{lemma}
The initial value problem:
\[
\left\{
\begin{array}
[c]{cc}%
D^{\frac{1}{2}}y(x)= & y(x)\\
y(x_{0})= & y_{0}%
\end{array}
\right.
\]
has a trivial solution $y\equiv0$ in $%
%TCIMACRO{\U{211d} }%
%BeginExpansion
\mathbb{R}
%EndExpansion
^{n}$ for $n\geq1$.
\end{lemma}

\begin{proof}
Taking the $D^{\frac{1}{2}}$ of both sides, we have
\[
\underset{\underset{Dy(x)}{\parallel}}{\underbrace{D^{\frac{1}{2}}\left(
D^{\frac{1}{2}}y(x)\right)  }}=D^{\frac{1}{2}}y(x)=y(x)
\]
That is
\[
Dy(x)=y(x)
\]
where $D=%
%TCIMACRO{\dsum \limits_{j=1}^{n}}%
%BeginExpansion
{\displaystyle\sum\limits_{j=1}^{n}}
%EndExpansion
e_{j}\frac{\partial}{\partial x_{j}}$

Now from
\[%
%TCIMACRO{\tciFourier}%
%BeginExpansion
\mathcal{F}%
%EndExpansion
\left(  \frac{\partial y}{\partial x_{j}}\right)  \left(  \xi\right)
=i\xi_{j}\widehat{f}(\xi)
\]
where $%
%TCIMACRO{\tciFourier}%
%BeginExpansion
\mathcal{F}%
%EndExpansion
$ or $\widehat{(.)}$ is the Fourier transform, and taking the Fourier
transform of both sides of
\[
Dy(x)=y(x)
\]
we have
\begin{align*}%
%TCIMACRO{\tciFourier}%
%BeginExpansion
\mathcal{F}%
%EndExpansion
\left(  Dy\right)  \left(  \xi\right)   & =%
%TCIMACRO{\dsum \limits_{j=1}^{n}}%
%BeginExpansion
{\displaystyle\sum\limits_{j=1}^{n}}
%EndExpansion
e_{j}%
%TCIMACRO{\tciFourier}%
%BeginExpansion
\mathcal{F}%
%EndExpansion
\left(  \frac{\partial y}{\partial x_{j}}\right)  \left(  \xi\right)  =i%
%TCIMACRO{\dsum \limits_{j=1}^{n}}%
%BeginExpansion
{\displaystyle\sum\limits_{j=1}^{n}}
%EndExpansion
e_{j}\xi_{j}\widehat{y}\left(  \xi\right)  =%
%TCIMACRO{\tciFourier}%
%BeginExpansion
\mathcal{F}%
%EndExpansion
\left(  y\right)  \left(  \xi\right)  =\widehat{y}(\xi)\\
& \Longrightarrow\text{ \ \ }i\left(  e.\xi\right)  \widehat{y}(\xi
)=\widehat{y}\left(  \xi\right) \\
& \Longrightarrow\text{ \ \ }\left(  i\left\langle e,\xi\right\rangle
-I\right)  \widehat{y}\left(  \xi\right)  =0.
\end{align*}

Now $i\left\langle e,\xi\right\rangle -I$ has eigen values $\pm i\parallel
\xi\parallel-1$ which will never be zero.

Therefore
\[
\widehat{y}\left(  \xi\right)  =0\Longrightarrow y\equiv0
\]

\end{proof}

\section{Orthogonality and Orthogonal Decompositions}

Let $d$ be a positive integer and $\Omega$ be a smooth and bounded domain in $%
%TCIMACRO{\U{211d} }%
%BeginExpansion
\mathbb{R}
%EndExpansion
^{d}$ with a non empty boundary $\partial\Omega$ and $p=2$

\begin{definition}
(Weak Fractional Derivative) For a function $f$, we say $g$ is the fractional
weak generalized $\frac{1}{2}^{th}$ order derivative of $f$ over $\Omega$
written as
\[
g=D^{\frac{1}{2}}f
\]
if%

\[
\
%TCIMACRO{\dint \limits_{\Omega}}%
%BeginExpansion
{\displaystyle\int\limits_{\Omega}}
%EndExpansion
g(x)\psi(x)d\Omega_{x}=%
%TCIMACRO{\dint \limits_{\Omega}}%
%BeginExpansion
{\displaystyle\int\limits_{\Omega}}
%EndExpansion
f(x)\left(  -\triangle\right)  ^{\frac{1}{4}}\psi(x)d\Omega_{x},\forall\psi\in
C_{0}^{\infty}\left(  \Omega\right)  \text{.}%
\]

\end{definition}

Clearly a function that is differentiable in the ordinary sense is also weakly
differentiable but not the converse. That is a weakly differentiable function
may not be differentiable in ordinary sense.

\begin{example}
Consider the function%
\[
f(x)=\left\{
\begin{array}
[c]{cc}%
0, & 0\leq x\leq1\\
x-1, & 1\leq x\leq2
\end{array}
\right.  .
\]

Then $f$ is continuous on $[0,1]$ but not differentiable in the regular sense
as the derivative%

\[
g\left(  x\right)  =\left\{
\begin{array}
[c]{cc}%
0, & 0<x<1\\
1, & 1<x<2
\end{array}
\right.
\]

is discontinuous at $x=1$. However $g$ is the first order weakly or
generalized derivative of $f$ on $[0,1]$ since%

\begin{align*}%
%TCIMACRO{\dint \limits_{\lbrack0,2]}}%
%BeginExpansion
{\displaystyle\int\limits_{\lbrack0,2]}}
%EndExpansion
f(x)\phi^{\prime}(x)dx  & =%
%TCIMACRO{\dint \limits_{[1,2]}}%
%BeginExpansion
{\displaystyle\int\limits_{[1,2]}}
%EndExpansion
x\phi^{\prime}(x)dx-%
%TCIMACRO{\dint \limits_{[1,2]}}%
%BeginExpansion
{\displaystyle\int\limits_{[1,2]}}
%EndExpansion
\phi^{\prime}(x)dx\\
& =x\phi\left(  x\right)  \mid_{1}^{2}-%
%TCIMACRO{\dint \limits_{[1,2]}}%
%BeginExpansion
{\displaystyle\int\limits_{[1,2]}}
%EndExpansion
\phi(x)dx-\phi\left(  2\right)  +\phi\left(  1\right) \\
& =-%
%TCIMACRO{\dint \limits_{[1,2]}}%
%BeginExpansion
{\displaystyle\int\limits_{[1,2]}}
%EndExpansion
\phi(x)dx\\
& =-%
%TCIMACRO{\dint \limits_{[0,2]}}%
%BeginExpansion
{\displaystyle\int\limits_{[0,2]}}
%EndExpansion
g(x)\phi(x)dx\\
& =-\int_{[0,2]}f^{\prime}(x)\phi(x)dx
\end{align*}
$\forall\phi\in C_{0}^{\infty}\left(  [0,2]\right)  $.
\end{example}

\begin{definition}
For $1<p<\infty$, $k\in%
%TCIMACRO{\U{2115} }%
%BeginExpansion
\mathbb{N}
%EndExpansion
\cup\{0\}$, the Sobolev space $W^{k,p}\left(  \Omega\right)  $ is defined as
the set of all functions $f$ in $%
%TCIMACRO{\tciLaplace}%
%BeginExpansion
\mathcal{L}%
%EndExpansion
^{p}\left(  \Omega\right)  $ such that the $\alpha^{\text{th }}$order weak
derivative $D^{\alpha}f\in%
%TCIMACRO{\tciLaplace}%
%BeginExpansion
\mathcal{L}%
%EndExpansion
^{p}\left(  \Omega\right)  $ for $0\leq\left\vert \alpha\right\vert \leq k$.
\end{definition}

As I mentioned in the preliminary, these function spaces are ideal spaces to
search for solutions to problems of reality unlike regular function spaces
such as $C^{k}$, where continuity to a particular order is required.

The particular square integrable Sobolev space (or Hilbert space of higher
regularity) $W^{k,2}\left(  \Omega\right)  $ where $p=2$, and $k\geq1$ is an
inner product space with inner product
\begin{equation}
\left\langle f,\text{ }g\right\rangle _{W^{k,2}(\Omega)}=%
%TCIMACRO{\dint \limits_{\Omega}}%
%BeginExpansion
{\displaystyle\int\limits_{\Omega}}
%EndExpansion
\left(
%TCIMACRO{\dsum \limits_{0\leq\mid\alpha\mid\leq k}}%
%BeginExpansion
{\displaystyle\sum\limits_{0\leq\mid\alpha\mid\leq k}}
%EndExpansion
D^{\alpha}f(x)D^{\alpha}g(x)\right)  d\Omega_{x}\label{innerprdct1}%
\end{equation}
with norm given by
\[
\Vert f\Vert_{W^{k,2}(\Omega)}=\left(  \left\langle f,\text{ }f\right\rangle
_{W^{k,2}(\Omega)}\right)  ^{\frac{1}{2}}%
\]
Therefore there is a distance or metric defined in terms of this norm given
by
\begin{equation}
\rho_{W^{k,2}\left(  \Omega\right)  }\left(  f,\text{ }g\right)  :=\Vert
f-g\Vert_{W^{k,2}\left(  \Omega\right)  }.\label{metric 1}%
\end{equation}
When $p=2$ and $k=0$, we have the usual Hilbert space
\[
W^{0,2}\left(  \Omega\right)  =%
%TCIMACRO{\tciLaplace}%
%BeginExpansion
\mathcal{L}%
%EndExpansion
^{2}\left(  \Omega\right)
\]
with inner product%

\begin{equation}
\left\langle f,\text{ }g\right\rangle _{W^{0,2}(\Omega)}=\left\langle f,\text{
}g\right\rangle _{%
%TCIMACRO{\tciLaplace}%
%BeginExpansion
\mathcal{L}%
%EndExpansion
^{2}(\Omega)}=%
%TCIMACRO{\dint \limits_{\Omega}}%
%BeginExpansion
{\displaystyle\int\limits_{\Omega}}
%EndExpansion
f\left(  x\right)  g(x)d\Omega_{x}\label{innerprodct2}%
\end{equation}

Now we define the fractional $\frac{1}{2}-$regularity order inner product in a
Hilbert space

\begin{definition}
When $k=1/2,p=2,$we define an inner product:
\begin{equation}
\left\langle f,\text{ }g\right\rangle _{W^{\frac{1}{2},2}(\Omega)}:=%
%TCIMACRO{\dint \limits_{\Omega}}%
%BeginExpansion
{\displaystyle\int\limits_{\Omega}}
%EndExpansion
f(x)g(x)d\Omega_{x}+%
%TCIMACRO{\dint \limits_{\Omega}}%
%BeginExpansion
{\displaystyle\int\limits_{\Omega}}
%EndExpansion%
%TCIMACRO{\dint \limits_{\Omega}}%
%BeginExpansion
{\displaystyle\int\limits_{\Omega}}
%EndExpansion
\frac{\left(  f(x)-f(y)\right)  (g(x)-g(y)}{\parallel x-y\parallel^{n+1}%
}d\Omega_{x}d\Omega_{y}.\label{fracinner}%
\end{equation}

and for $f\in W^{\frac{1}{2},2}(\Omega),$the $W^{\frac{1}{2},2}(\Omega)$ norm
is defined by
\[
\parallel f\parallel_{W^{\frac{1}{2},2}(\Omega)}^{2}:=%
%TCIMACRO{\dint \limits_{\Omega}}%
%BeginExpansion
{\displaystyle\int\limits_{\Omega}}
%EndExpansion
f^{2}(x)d\Omega_{x}+%
%TCIMACRO{\dint \limits_{\Omega}}%
%BeginExpansion
{\displaystyle\int\limits_{\Omega}}
%EndExpansion%
%TCIMACRO{\dint \limits_{\Omega}}%
%BeginExpansion
{\displaystyle\int\limits_{\Omega}}
%EndExpansion
\frac{\left(  f(x)-f(y)\right)  ^{2}}{\parallel x-y\parallel^{n+1}}d\Omega
_{x}d\Omega_{y}%
\]

\[
\therefore\text{ \ \ \ \ \ }\parallel f\parallel_{W^{\frac{1}{2},2}(\Omega
)}=\sqrt{\left\langle f,\text{ }f\right\rangle _{W^{\frac{1}{2},2}(\Omega)}%
}=\sqrt{%
%TCIMACRO{\dint \limits_{\Omega}}%
%BeginExpansion
{\displaystyle\int\limits_{\Omega}}
%EndExpansion
f^{2}(x)d\Omega_{x}+%
%TCIMACRO{\dint \limits_{\Omega}}%
%BeginExpansion
{\displaystyle\int\limits_{\Omega}}
%EndExpansion%
%TCIMACRO{\dint \limits_{\Omega}}%
%BeginExpansion
{\displaystyle\int\limits_{\Omega}}
%EndExpansion
\frac{\left(  f(x)-f(y)\right)  ^{2}}{\parallel x-y\parallel^{n+1}}d\Omega
_{x}d\Omega_{y}}%
\]

\end{definition}

\begin{definition}
We say two functions $f,$ $g\in W^{\frac{1}{2},2}\left(  \Omega\right)  $
orthogonal with respect to the inner product defined by \ref{fracinner} if
\begin{equation}
\left\langle f,\text{ }g\right\rangle _{W^{\frac{1}{2},2}(\Omega
)}=0.\label{orthogonality}%
\end{equation}

\end{definition}

For more orthogonal functions of integer regularity, see $[2]$.

\ \ \ \ 

We investigate orthogonal decompositions of function spaces with inner product
$\left\langle ,\right\rangle $ so that any function in the space is an
orthogonal sum of component functions from the orthogonal parts of the space.
The spaces we considered in earlier studies are the Sobolev spaces
$W^{k-1,2}\left(  \Omega\right)  $, the space for $p=2$ and $k\in%
%TCIMACRO{\U{2115} }%
%BeginExpansion
\mathbb{N}
%EndExpansion
$ and established several properties in \cite{dlakew1} and \cite{dlakew2}.

We consider now fractional decompositions of Hilbert spaces of fractional regularities:

\begin{proposition}
The space $%
%TCIMACRO{\tciLaplace}%
%BeginExpansion
\mathcal{L}%
%EndExpansion
^{2}\left(  \Omega\right)  $ has an orthogonal decomposition interms of the
Caputo fractional $D^{\frac{1}{2}}:$
\begin{equation}%
%TCIMACRO{\tciLaplace}%
%BeginExpansion
\mathcal{L}%
%EndExpansion
^{2}\left(  \Omega\right)  =A^{\frac{1}{2},2}\left(  \Omega\right)  \oplus
D^{\frac{1}{2}}\left(  W_{0}^{\frac{1}{2},2}\left(  \Omega\right)  \right)
\label{decomp1}%
\end{equation}
where
\[
A^{\frac{1}{2},2}\left(  \Omega\right)  =KerD^{\frac{1}{2}}\cap%
%TCIMACRO{\tciLaplace}%
%BeginExpansion
\mathcal{L}%
%EndExpansion
^{2}\left(  \Omega\right)  =\{g(x)=%
%TCIMACRO{\dint \limits_{\partial\Omega}}%
%BeginExpansion
{\displaystyle\int\limits_{\partial\Omega}}
%EndExpansion
\psi_{\frac{1}{2}}\left(  y-x\right)  \upsilon_{\frac{1}{2}}\left(  y\right)
h(y)d\partial\Omega_{y}:h\in%
%TCIMACRO{\tciLaplace}%
%BeginExpansion
\mathcal{L}%
%EndExpansion
^{2}\left(  \Omega\right)  \}
\]
and%
\[
W_{0}^{\frac{1}{2},2}\left(  \Omega\right)  =\{f\in W^{\frac{1}{2},2}\left(
\Omega\right)  :\left(  f_{\mid\partial\Omega},D^{\frac{1}{2}}f_{\mid
\partial\Omega}\right)  =\left(  0,0\right)  \}
\]
so that
\[
\forall f\in%
%TCIMACRO{\tciLaplace}%
%BeginExpansion
\mathcal{L}%
%EndExpansion
^{2}\left(  \Omega\right)  ,\exists g\in A^{\frac{1}{2},2}\left(
\Omega\right)  \text{ \ and \ }h\in D^{\frac{1}{2}}\left(  W_{0}^{\frac{1}%
{2},2}\left(  \Omega\right)  \right)
\]%
\[
f=g\uplus h.
\]

\end{proposition}

\begin{proof}
Let $f\in%
%TCIMACRO{\tciLaplace}%
%BeginExpansion
\mathcal{L}%
%EndExpansion
^{2}\left(  \Omega\right)  $. Then consider $g(x)=\frac{1}{\mu\left(
\Omega\right)  }%
%TCIMACRO{\dint \limits_{\Omega}}%
%BeginExpansion
{\displaystyle\int\limits_{\Omega}}
%EndExpansion
f(x)d\Omega_{x}$ which is a constant. Then $g\in KerD^{\frac{1}{2}}\left(
\Omega\right)  \cap%
%TCIMACRO{\tciLaplace}%
%BeginExpansion
\mathcal{L}%
%EndExpansion
^{2}\left(  \Omega\right)  $. Now consider $h=f-g$. Claim that $h\in
D^{\frac{1}{2}}\left(  W_{0}^{\frac{1}{2},2}\left(  \Omega\right)  \right)  $

Therefore
\[
\exists\eta\in W_{0}^{\frac{1}{2},2}\left(  \Omega\right)  \ni h\left(
x\right)  =D^{\frac{1}{2}}\eta\left(  x\right)  =\left(  -\triangle\right)
^{\frac{1}{4}}\eta\left(  x\right)  =\frac{\Gamma\left(  \frac{n+1}{2}\right)
}{\sqrt{\pi^{n+1}}}P.V.%
%TCIMACRO{\dint \limits_{\Omega}}%
%BeginExpansion
{\displaystyle\int\limits_{\Omega}}
%EndExpansion
\frac{\eta\left(  x\right)  -\eta\left(  y\right)  }{\parallel x-y\parallel
^{n+1}}d\Omega_{y}%
\]
.

Then
\[
\left\langle g(x),h(x)\right\rangle _{%
%TCIMACRO{\tciLaplace}%
%BeginExpansion
\mathcal{L}%
%EndExpansion
^{2}\left(  \Omega\right)  }=\left\langle g(x),D^{\frac{1}{2}}\eta
(x)\right\rangle _{%
%TCIMACRO{\tciLaplace}%
%BeginExpansion
\mathcal{L}%
%EndExpansion
^{2}\left(  \Omega\right)  }=\left\langle D^{\frac{1}{2}}%
g(x),h(x)\right\rangle _{%
%TCIMACRO{\tciLaplace}%
%BeginExpansion
\mathcal{L}%
%EndExpansion
^{2}\left(  \Omega\right)  }=\left\langle 0,h(x)\right\rangle _{%
%TCIMACRO{\tciLaplace}%
%BeginExpansion
\mathcal{L}%
%EndExpansion
^{2}\left(  \Omega\right)  }=0,\because g\text{ is a constant}%
\]

Hence,%
\[
f=g+h\text{ \ with property}:\left\langle g,h\right\rangle _{%
%TCIMACRO{\tciLaplace}%
%BeginExpansion
\mathcal{L}%
%EndExpansion
^{2}\left(  \Omega\right)  }=0\Longrightarrow f=g\uplus h
\]

\end{proof}

\begin{remark}
Constants are also in $KerD^{\frac{1}{2}}\left(  \Omega\right)  \cap%
%TCIMACRO{\tciLaplace}%
%BeginExpansion
\mathcal{L}%
%EndExpansion
^{2}\left(  \Omega\right)  $ as long as $\Omega$ is bounded, since for $%
%TCIMACRO{\dint \limits_{\partial\Omega}}%
%BeginExpansion
{\displaystyle\int\limits_{\partial\Omega}}
%EndExpansion
\psi_{\frac{1}{2}}\left(  y-x\right)  \upsilon_{\frac{1}{2}}\left(  y\right)
cd\partial\Omega_{y}=c\underset{\underset{1}{\parallel}}{\underbrace{%
%TCIMACRO{\dint \limits_{\partial\Omega}}%
%BeginExpansion
{\displaystyle\int\limits_{\partial\Omega}}
%EndExpansion
\psi_{\frac{1}{2}}\left(  y-x\right)  \upsilon_{\frac{1}{2}}\left(  y\right)
d\partial\Omega_{y}}}=c$
\end{remark}

\begin{lemma}
For $\psi_{\frac{1}{2}}\left(  y-x\right)  $, the fundamental solution to
$D^{\frac{1}{2}},%
%TCIMACRO{\dint \limits_{\partial\Omega}}%
%BeginExpansion
{\displaystyle\int\limits_{\partial\Omega}}
%EndExpansion
\psi_{\frac{1}{2}}\left(  y-x\right)  \upsilon_{\frac{1}{2}}\left(  y\right)
d\partial\Omega_{y}=1.$
\end{lemma}

\begin{proof}
For $y\in\partial\Omega,$and $x\in\Omega\setminus\partial\Omega$ consider a
small ball $B_{\delta}\left(  x\right)  \subseteq\Omega$ of small radius
$\delta=\parallel x-y\parallel>0.$ Consider the normal vector $\nu_{\frac
{1}{2}}\left(  y\right)  =\frac{y-x}{\delta}$.

Then
\[%
%TCIMACRO{\dint \limits_{\partial B_{\delta}\left(  x\right)  }}%
%BeginExpansion
{\displaystyle\int\limits_{\partial B_{\delta}\left(  x\right)  }}
%EndExpansion
\psi_{\frac{1}{2}}\left(  y-x\right)  \upsilon_{\frac{1}{2}}\left(  y\right)
d\partial\Omega_{y}=%
%TCIMACRO{\dint \limits_{\partial B_{\delta}\left(  x\right)  }}%
%BeginExpansion
{\displaystyle\int\limits_{\partial B_{\delta}\left(  x\right)  }}
%EndExpansion
C_{n,\frac{1}{2}}\frac{1}{\delta^{n-\frac{1}{2}}}.\delta^{n-1}d\omega
_{n}=C_{n,\frac{1}{2}}%
%TCIMACRO{\dint \limits_{\partial B_{1}\left(  0\right)  }}%
%BeginExpansion
{\displaystyle\int\limits_{\partial B_{1}\left(  0\right)  }}
%EndExpansion
\delta^{\frac{-1}{2}}d\omega_{n}=C_{n,\frac{1}{2}}\delta^{\frac{-1}{2}}%
\omega_{n}=1
\]
$\because C_{n,\frac{1}{2}}=\frac{\sqrt{\delta}}{\omega_{n}}$, where
$\omega_{n}:=$ surface area of a unit sphere in $%
%TCIMACRO{\U{211d} }%
%BeginExpansion
\mathbb{R}
%EndExpansion
^{n}$.
\end{proof}

\begin{corollary}%
\[
KerD^{\frac{1}{2}}\left(  \Omega\right)  =\{%
%TCIMACRO{\dint \limits_{\partial\Omega}}%
%BeginExpansion
{\displaystyle\int\limits_{\partial\Omega}}
%EndExpansion
\psi_{\frac{1}{2}}\left(  y-x\right)  \upsilon_{\frac{1}{2}}\left(  y\right)
g(y)d\partial\Omega_{y}%
\]
for $g\in%
%TCIMACRO{\tciLaplace}%
%BeginExpansion
\mathcal{L}%
%EndExpansion
^{2}\left(  \partial\Omega\right)  \}$.
\end{corollary}

\begin{proof}
Clearly, for $g\in%
%TCIMACRO{\tciLaplace}%
%BeginExpansion
\mathcal{L}%
%EndExpansion
^{2}\left(  \partial\Omega\right)  \},$let $u$ be its smooth extension to
$\Omega$. Then
\[
u(x)=%
%TCIMACRO{\dint \limits_{\partial\Omega}}%
%BeginExpansion
{\displaystyle\int\limits_{\partial\Omega}}
%EndExpansion
\psi_{\frac{1}{2}}\left(  y-x\right)  g(y)\upsilon_{\frac{1}{2}}\left(
y\right)  d\partial\Omega_{y}+%
%TCIMACRO{\dint \limits_{\Omega}}%
%BeginExpansion
{\displaystyle\int\limits_{\Omega}}
%EndExpansion
\psi_{\frac{1}{2}}\left(  y-x\right)  D^{\frac{1}{2}}u\left(  y\right)
d\Omega_{y}%
\]

Then differentiating both sides with half order:
\begin{align*}
D^{\frac{1}{2}}u(x)  & =D^{\frac{1}{2}}\{%
%TCIMACRO{\dint \limits_{\partial\Omega}}%
%BeginExpansion
{\displaystyle\int\limits_{\partial\Omega}}
%EndExpansion
\psi_{\frac{1}{2}}\left(  y-x\right)  g(y)\upsilon_{\frac{1}{2}}\left(
y\right)  d\partial\Omega_{y}+%
%TCIMACRO{\dint \limits_{\Omega}}%
%BeginExpansion
{\displaystyle\int\limits_{\Omega}}
%EndExpansion
\psi_{\frac{1}{2}}\left(  y-x\right)  D^{\frac{1}{2}}u\left(  y\right)
d\Omega_{y}\}\\
& =D^{\frac{1}{2}}\{%
%TCIMACRO{\dint \limits_{\partial\Omega}}%
%BeginExpansion
{\displaystyle\int\limits_{\partial\Omega}}
%EndExpansion
\psi_{\frac{1}{2}}\left(  y-x\right)  g(y)\upsilon_{\frac{1}{2}}\left(
y\right)  d\partial\Omega_{y}\}+D^{\frac{1}{2}}\{%
%TCIMACRO{\dint \limits_{\Omega}}%
%BeginExpansion
{\displaystyle\int\limits_{\Omega}}
%EndExpansion
\psi_{\frac{1}{2}}\left(  y-x\right)  D^{\frac{1}{2}}u\left(  y\right)
d\Omega_{y}\}
\end{align*}

But
\[
D^{\frac{1}{2}}\{%
%TCIMACRO{\dint \limits_{\Omega}}%
%BeginExpansion
{\displaystyle\int\limits_{\Omega}}
%EndExpansion
\psi_{\frac{1}{2}}\left(  y-x\right)  D^{\frac{1}{2}}u\left(  y\right)
d\Omega_{y}\}=D^{\frac{1}{2}}\xi_{\Omega}^{\frac{1}{2}}(D^{\frac{1}{2}%
}u(x))=D^{\frac{1}{2}}u(x)
\]
since $D^{\frac{1}{2}}\xi_{\Omega}^{\frac{1}{2}}$ is the identity operator.

Hence,
\[
D^{\frac{1}{2}}u(x)=D^{\frac{1}{2}}\{%
%TCIMACRO{\dint \limits_{\partial\Omega}}%
%BeginExpansion
{\displaystyle\int\limits_{\partial\Omega}}
%EndExpansion
\psi_{\frac{1}{2}}\left(  y-x\right)  g(y)\upsilon_{\frac{1}{2}}\left(
y\right)  d\partial\Omega_{y}\}+D^{\frac{1}{2}}u(x)
\]

\[
\Longrightarrow\ \ \ \ \ \ \ \ \ \ \ \ D^{\frac{1}{2}}\{%
%TCIMACRO{\dint \limits_{\partial\Omega}}%
%BeginExpansion
{\displaystyle\int\limits_{\partial\Omega}}
%EndExpansion
\psi_{\frac{1}{2}}\left(  y-x\right)  g(y)\upsilon_{\frac{1}{2}}\left(
y\right)  d\partial\Omega_{y}\}=0,\text{ \ }i.e.,%
%TCIMACRO{\dint \limits_{\partial\Omega}}%
%BeginExpansion
{\displaystyle\int\limits_{\partial\Omega}}
%EndExpansion
\psi_{\frac{1}{2}}\left(  y-x\right)  g(y)\upsilon_{\frac{1}{2}}\left(
y\right)  d\partial\Omega_{y}\in\ker D^{\frac{1}{2}}\left(  \Omega\right)  .
\]

This result includes constants as well, i.e., $g(x)=c$
\end{proof}

\begin{example}
Consider $f(x)=\Vert x\Vert^{2}=%
%TCIMACRO{\dsum \limits_{i=1}^{n}}%
%BeginExpansion
{\displaystyle\sum\limits_{i=1}^{n}}
%EndExpansion
x_{i}^{2}:\Omega:=\Pi_{i=1}^{n}\left[  -1,1\right]  \longrightarrow%
%TCIMACRO{\U{211d} }%
%BeginExpansion
\mathbb{R}
%EndExpansion
$. Consider
\[
g\left(  x\right)  =\frac{1}{|\Omega|}%
%TCIMACRO{\dint \limits_{\Omega}}%
%BeginExpansion
{\displaystyle\int\limits_{\Omega}}
%EndExpansion
f(x)d\Omega_{x}=\frac{1}{2^{n}}%
%TCIMACRO{\dint \limits_{-1}^{1}}%
%BeginExpansion
{\displaystyle\int\limits_{-1}^{1}}
%EndExpansion
...%
%TCIMACRO{\dint \limits_{-1}^{1}}%
%BeginExpansion
{\displaystyle\int\limits_{-1}^{1}}
%EndExpansion%
%TCIMACRO{\dsum \limits_{i=1}^{n}}%
%BeginExpansion
{\displaystyle\sum\limits_{i=1}^{n}}
%EndExpansion
x_{i}^{2}\underset{d\Omega_{x}}{\underbrace{dx_{1}dx_{2}...dx_{n}}}=\frac
{1}{2^{n}}\left(  n.\frac{2^{n}}{3}\right)  =\frac{n}{3}%
\]

Set $h\left(  x\right)  =\Vert x\Vert^{2}-\frac{n}{3}$. Then $f(x)=g(x)+h(x)$
and
\[
\left\langle g(x),h(x)\right\rangle _{%
%TCIMACRO{\tciLaplace}%
%BeginExpansion
\mathcal{L}%
%EndExpansion
^{2}\left(  \Omega\right)  }=%
%TCIMACRO{\dint \limits_{\Omega}}%
%BeginExpansion
{\displaystyle\int\limits_{\Omega}}
%EndExpansion
g(x)h(x)d\Omega_{x}=\frac{n}{3}%
%TCIMACRO{\dint \limits_{\Omega}}%
%BeginExpansion
{\displaystyle\int\limits_{\Omega}}
%EndExpansion
\left(  \Vert x\Vert^{2}-\frac{n}{3}\right)  d\Omega_{x}=\frac{n}{3}\left(
\frac{n2^{n}}{3}-\frac{n2^{n}}{3}\right)  =0
\]

Hence
\[
f(x)=g(x)\uplus h(x)
\]

Also the parallelogram law holds:%
\[
\Vert f\Vert_{%
%TCIMACRO{\tciLaplace}%
%BeginExpansion
\mathcal{L}%
%EndExpansion
^{2}\left(  \Omega\right)  }=\left(  \Vert g\Vert_{%
%TCIMACRO{\tciLaplace}%
%BeginExpansion
\mathcal{L}%
%EndExpansion
^{2}\left(  \Omega\right)  }^{2}+\Vert h\Vert_{%
%TCIMACRO{\tciLaplace}%
%BeginExpansion
\mathcal{L}%
%EndExpansion
^{2}\left(  \Omega\right)  }^{2}\right)  ^{\frac{1}{2}}%
\]

\end{example}

These decompositions enable us to give the following main results of our research.

\section{Boundary Value Problems and Decompositions}

In this section we consider first and second order boundary value problems and
see how the solutions are simply the orthogonal sums of parts that evolve from
boundary values of the solution and interior values of the derivative of the
solution to the given order. We start with the first order Cauchy problem.

\begin{theorem}
\begin{proposition}
Let $f\in%
%TCIMACRO{\tciLaplace}%
%BeginExpansion
\mathcal{L}%
%EndExpansion
^{2}\left(  \Omega\right)  $,$g\in%
%TCIMACRO{\tciLaplace}%
%BeginExpansion
\mathcal{L}%
%EndExpansion
^{2}\left(  \partial\Omega\right)  $ The first Caputo fractional $1/2$ order
Cauchy problem%
\begin{equation}
\left\{
\begin{array}
[c]{cc}%
D^{\frac{1}{2}}u=f & \text{in }\Omega\\
u=g & \text{on }\partial\Omega
\end{array}
\right.
\end{equation}
has a solution $u\in W^{\frac{1}{2},2}\left(  \Omega\right)  $ such that
\[
u=\left[  u\right]  _{f}\uplus\left[  u\right]  _{g}%
\]
where $\left[  u\right]  _{f}$ is the part of the solution that evolves from
$f $ and $\left[  u\right]  _{g}$ is the part of the solution that evolves
from $g $ with the following properties

$(i)$
\[
\left\langle \left[  u\right]  _{f},\text{ }\left[  u\right]  _{g}%
\right\rangle _{W^{\frac{1}{2},2}\left(  \Omega\right)  }=0
\]
$\left(  ii\right)  $%
\[
\parallel u\parallel_{W^{\frac{1}{2},2}\left(  \Omega\right)  }^{2}%
=\parallel\left[  u\right]  _{f}\parallel_{W^{\frac{1}{2},2}\left(
\Omega\right)  }^{2}+\parallel\left[  u\right]  _{g}\parallel_{W^{\frac{1}%
{2},2}\left(  \Omega\right)  }^{2}%
\]

\end{proposition}

$\left(  iii\right)  $
\[
\left\langle u,\text{ }\left[  u\right]  _{f}\right\rangle _{W^{\frac{1}{2}%
,2}\left(  \Omega\right)  }=\parallel\left[  u\right]  _{f}\parallel
_{W^{\frac{1}{2},2}\left(  \Omega\right)  }^{2}%
\]

$\left(  iv\right)  $
\[
\left\langle u,\text{ }\left[  u\right]  _{g}\right\rangle _{W^{\frac{1}{2}%
,2}\left(  \Omega\right)  }=\parallel\left[  u\right]  _{g}\parallel
_{W^{\frac{1}{2},2}\left(  \Omega\right)  }^{2}%
\]

\end{theorem}

\begin{proof}
For two functions $f,$ $g\in C^{1}\left(  \Omega\right)  $, integration by
parts provide%

\[%
%TCIMACRO{\dint \limits_{\Omega}}%
%BeginExpansion
{\displaystyle\int\limits_{\Omega}}
%EndExpansion
f\left(  y-x\right)  D_{y}^{\frac{1}{2}}g(y)d\Omega_{y}=%
%TCIMACRO{\dint \limits_{\partial\Omega}}%
%BeginExpansion
{\displaystyle\int\limits_{\partial\Omega}}
%EndExpansion
f\left(  y-x\right)  \upsilon_{\frac{1}{2}}\left(  y\right)  g\left(
y\right)  d\partial\Omega_{y}-%
%TCIMACRO{\dint \limits_{\Omega}}%
%BeginExpansion
{\displaystyle\int\limits_{\Omega}}
%EndExpansion
D_{y}^{\frac{1}{2}}f\left(  y-x\right)  g\left(  y\right)  d\Omega_{y}.
\]
Then taking $f=\Gamma_{\frac{1}{2}}$ and $g=u$, we have
\[%
%TCIMACRO{\dint \limits_{\Omega}}%
%BeginExpansion
{\displaystyle\int\limits_{\Omega}}
%EndExpansion
\Gamma_{\frac{1}{2}}\left(  y-x\right)  D_{y}^{\frac{1}{2}}u(y)d\Omega_{y}=%
%TCIMACRO{\dint \limits_{\partial\Omega}}%
%BeginExpansion
{\displaystyle\int\limits_{\partial\Omega}}
%EndExpansion
\Gamma_{\frac{1}{2}}\left(  y-x\right)  \upsilon_{\frac{1}{2}}\left(
y\right)  u\left(  y\right)  d\partial\Omega_{y}-%
%TCIMACRO{\dint \limits_{\Omega}}%
%BeginExpansion
{\displaystyle\int\limits_{\Omega}}
%EndExpansion
D_{y}^{\frac{1}{2}}\Gamma_{\frac{1}{2}}\left(  y-x\right)  u\left(  y\right)
d\Omega_{y}%
\]
where $\Gamma_{\frac{1}{2}}$ is the fundamental solution to the fractional
half Dirac operator $D^{\frac{1}{2}}$.

But
\[
D_{y}^{\frac{1}{2}}\Gamma_{\frac{1}{2}}\left(  y-x\right)  =\delta\left(
y-x\right)
\]
the Kronecker delta function and thus, \ %

\begin{align*}%
%TCIMACRO{\dint \limits_{\Omega}}%
%BeginExpansion
{\displaystyle\int\limits_{\Omega}}
%EndExpansion
\Gamma_{\frac{1}{2}}\left(  y-x\right)  D^{\frac{1}{2}}u(y)d\Omega_{y}  & =%
%TCIMACRO{\dint \limits_{\partial\Omega}}%
%BeginExpansion
{\displaystyle\int\limits_{\partial\Omega}}
%EndExpansion
\Gamma_{\frac{1}{2}}\left(  y-x\right)  \upsilon_{\frac{1}{2}}\left(
y\right)  u\left(  y\right)  d\partial\Omega_{y}-%
%TCIMACRO{\dint \limits_{\Omega}}%
%BeginExpansion
{\displaystyle\int\limits_{\Omega}}
%EndExpansion
\delta\left(  y-x\right)  u\left(  y\right)  d\Omega_{y}\\
& =%
%TCIMACRO{\dint \limits_{\partial\Omega}}%
%BeginExpansion
{\displaystyle\int\limits_{\partial\Omega}}
%EndExpansion
\Gamma_{\frac{1}{2}}\left(  y-x\right)  \upsilon_{\frac{1}{2}}\left(
y\right)  u\left(  y\right)  d\partial\Omega_{y}-u(x)
\end{align*}

i.e.,
\begin{equation}
u\left(  x\right)  =%
%TCIMACRO{\dint \limits_{\partial\Omega}}%
%BeginExpansion
{\displaystyle\int\limits_{\partial\Omega}}
%EndExpansion
\Gamma_{\frac{1}{2}}\left(  y-x\right)  \upsilon_{\frac{1}{2}}\left(
y\right)  u\left(  y\right)  d\partial\Omega_{y}-%
%TCIMACRO{\dint \limits_{\Omega}}%
%BeginExpansion
{\displaystyle\int\limits_{\Omega}}
%EndExpansion
\Gamma_{\frac{1}{2}}\left(  y-x\right)  D^{\frac{1}{2}}u(y)d\Omega
_{y}\label{borel-pompeiu}%
\end{equation}
which provides the integral representation of the solution $u$ to the BVP
given by
\[
u\left(  x\right)  =%
%TCIMACRO{\dint \limits_{\partial\Omega}}%
%BeginExpansion
{\displaystyle\int\limits_{\partial\Omega}}
%EndExpansion
\Gamma_{\frac{1}{2}}\left(  y-x\right)  \upsilon_{\frac{1}{2}}\left(
y\right)  g(y)d\partial\Omega_{y}+\left(  -%
%TCIMACRO{\dint \limits_{\Omega}}%
%BeginExpansion
{\displaystyle\int\limits_{\Omega}}
%EndExpansion
\Gamma_{\frac{1}{2}}\left(  y-x\right)  f(y)d\Omega_{y}\right)  .
\]
We will see that this sum is in fact an orthogonal sum $\uplus$. \ 

Because the solution
\[
u\in W^{\frac{1}{2},2}\left(  \Omega\right)  =A^{\frac{1}{2},2}\left(
\Omega\right)  \oplus D^{\frac{1}{2}}\left(  W_{0}^{\frac{1}{2},2}\left(
\Omega\right)  \right)
\]
has a unique decomposition as sum of components from the two sub spaces,
$A^{\frac{1}{2},2}\left(  \Omega\right)  $ and $D^{\frac{1}{2}}\left(
W_{0}^{\frac{3}{2},2}\left(  \Omega\right)  \right)  $ and the first integral
\[%
%TCIMACRO{\dint \limits_{\partial\Omega}}%
%BeginExpansion
{\displaystyle\int\limits_{\partial\Omega}}
%EndExpansion
\Gamma_{\frac{1}{2}}\left(  y-x\right)  \upsilon_{\frac{1}{2}}\left(
y\right)  g(y)d\partial\Omega_{y}%
\]
is monogenic over $\Omega$ and hence an element of
\[
KerD^{\frac{1}{2}}\cap%
%TCIMACRO{\tciLaplace}%
%BeginExpansion
\mathcal{L}%
%EndExpansion
^{2}\left(  \Omega\right)  =A^{\frac{1}{2},2}\left(  \Omega\right)  .
\]
We need to verify that the second integral%

\[
\left(  -%
%TCIMACRO{\dint \limits_{\Omega}}%
%BeginExpansion
{\displaystyle\int\limits_{\Omega}}
%EndExpansion
\Gamma_{\frac{1}{2}}\left(  y-x\right)  f(y)d\Omega_{y}\right)  \in
D^{\frac{1}{2}}\left(  W_{0}^{\frac{1}{2},2}\left(  \Omega\right)  \right)
\]
as well. That is, there is a function $\xi\in W_{0}^{\frac{3}{2},2}\left(
\Omega\right)  $ so that%

\[
-%
%TCIMACRO{\dint \limits_{\Omega}}%
%BeginExpansion
{\displaystyle\int\limits_{\Omega}}
%EndExpansion
\Gamma_{\frac{1}{2}}\left(  y-x\right)  f(y)d\Omega_{y}=D^{\frac{1}{2}}%
\xi\left(  x\right)  \text{ \ \ with \ \ }\xi_{\mid\partial\Omega}=0.
\]
Clearly from the fact that \ \ %

\[
u\left(  x\right)  =%
%TCIMACRO{\dint \limits_{\partial\Omega}}%
%BeginExpansion
{\displaystyle\int\limits_{\partial\Omega}}
%EndExpansion
\Gamma_{\frac{1}{2}}\left(  y-x\right)  \upsilon_{\frac{1}{2}}\left(
y\right)  g(y)d\partial\Omega_{y}+\left(  -%
%TCIMACRO{\dint \limits_{\Omega}}%
%BeginExpansion
{\displaystyle\int\limits_{\Omega}}
%EndExpansion
\Gamma_{\frac{1}{2}}\left(  y-x\right)  f(y)d\Omega_{y}\right)  \text{ \ and
\ }u_{\mid\partial\Omega}=g
\]
we have, the trace $\tau$ of $u$ on $\partial\Omega$%

\begin{align*}
\tau u_{\mid\partial\Omega}  & =\tau\left(
%TCIMACRO{\dint \limits_{\partial\Omega}}%
%BeginExpansion
{\displaystyle\int\limits_{\partial\Omega}}
%EndExpansion
\Gamma_{\frac{1}{2}}\left(  y-x\right)  \upsilon_{\frac{1}{2}}\left(
y\right)  g(y)d\partial\Omega_{y}+\left(  -%
%TCIMACRO{\dint \limits_{\Omega}}%
%BeginExpansion
{\displaystyle\int\limits_{\Omega}}
%EndExpansion
\Gamma_{\frac{1}{2}}\left(  y-x\right)  f(y)d\Omega_{y}\right)  \right)
_{\mid\partial\Omega}\\
& =\tau\left(
%TCIMACRO{\dint \limits_{\partial\Omega}}%
%BeginExpansion
{\displaystyle\int\limits_{\partial\Omega}}
%EndExpansion
\Gamma_{\frac{1}{2}}\left(  y-x\right)  \upsilon_{\frac{1}{2}}\left(
y\right)  g(y)d\partial\Omega_{y}\right)  _{\mid\partial\Omega}+\tau\left(  -%
%TCIMACRO{\dint \limits_{\Omega}}%
%BeginExpansion
{\displaystyle\int\limits_{\Omega}}
%EndExpansion
\Gamma_{\frac{1}{2}}\left(  y-x\right)  f(y)d\Omega_{y}\right)  _{\mid
\partial\Omega}\\
& =g
\end{align*}

\[
\Longrightarrow\text{ \ \ \ \ }\tau\left(  -%
%TCIMACRO{\dint \limits_{\Omega}}%
%BeginExpansion
{\displaystyle\int\limits_{\Omega}}
%EndExpansion
\Gamma_{\frac{1}{2}}\left(  y-x\right)  f(y)d\Omega_{y}\right)  =\left(  -%
%TCIMACRO{\dint \limits_{\Omega}}%
%BeginExpansion
{\displaystyle\int\limits_{\Omega}}
%EndExpansion
\Gamma_{\frac{1}{2}}\left(  y-x\right)  f(y)d\Omega_{y}\right)  _{\mid
\partial\Omega}=0\text{.}%
\]

\ \ 

Thus $\xi_{\mid\partial\Omega}=0$. In an analogous manner of the integral
representation of the solution $u$, we get the integral representation of
$\xi$ as

\
\begin{align*}
\xi\left(  x\right)   & =-%
%TCIMACRO{\dint \limits_{\Omega}}%
%BeginExpansion
{\displaystyle\int\limits_{\Omega}}
%EndExpansion
\Gamma_{\frac{1}{2}}\left(  z-x\right)
%TCIMACRO{\dint \limits_{\Omega}}%
%BeginExpansion
{\displaystyle\int\limits_{\Omega}}
%EndExpansion
\Gamma_{\frac{1}{2}}\left(  y-z\right)  f(y)d\Omega_{y}d\Omega_{z}\\
& =-%
%TCIMACRO{\dint \limits_{\Omega}}%
%BeginExpansion
{\displaystyle\int\limits_{\Omega}}
%EndExpansion%
%TCIMACRO{\dint \limits_{\Omega}}%
%BeginExpansion
{\displaystyle\int\limits_{\Omega}}
%EndExpansion
\Gamma_{\frac{1}{2}}\left(  z-x\right)  \Gamma_{\frac{1}{2}}\left(
y-z\right)  f(y)d\Omega_{y}d\Omega_{z}\in W_{0}^{1,2}\left(  \Omega\right)
\text{.}%
\end{align*}
Therefore, there is a \ %

\[
\xi\left(  x\right)  =-%
%TCIMACRO{\dint \limits_{\Omega}}%
%BeginExpansion
{\displaystyle\int\limits_{\Omega}}
%EndExpansion%
%TCIMACRO{\dint \limits_{\Omega}}%
%BeginExpansion
{\displaystyle\int\limits_{\Omega}}
%EndExpansion
\Gamma_{\frac{1}{2}}\left(  z-x\right)  \Gamma_{\frac{1}{2}}\left(
y-z\right)  f(y)d\Omega_{y}d\Omega_{z}\in W_{0}^{1,2}\left(  \Omega\right)
\]
so that \ \ \ %

\begin{align*}
D_{x}^{\frac{1}{2}}\xi\left(  x\right)   & =D_{x}^{\frac{1}{2}}\left(  -%
%TCIMACRO{\dint \limits_{\Omega}}%
%BeginExpansion
{\displaystyle\int\limits_{\Omega}}
%EndExpansion%
%TCIMACRO{\dint \limits_{\Omega}}%
%BeginExpansion
{\displaystyle\int\limits_{\Omega}}
%EndExpansion
\Gamma_{\frac{1}{2}}\left(  z-x\right)  \Gamma_{\frac{1}{2}}\left(
y-z\right)  f(y)d\Omega_{y}d\Omega_{z}\right) \\
& =-%
%TCIMACRO{\dint \limits_{\Omega}}%
%BeginExpansion
{\displaystyle\int\limits_{\Omega}}
%EndExpansion
\Gamma_{\frac{1}{2}}\left(  y-x\right)  f(y)d\Omega_{y}\text{.}%
\end{align*}
Hence \ %

\begin{align*}
u(x)  & =%
%TCIMACRO{\dint \limits_{\partial\Omega}}%
%BeginExpansion
{\displaystyle\int\limits_{\partial\Omega}}
%EndExpansion
\Gamma_{\frac{1}{2}}\left(  y-x\right)  \upsilon_{\frac{1}{2}}\left(
y\right)  g(y)d\partial\Omega_{y}\uplus\left(  -%
%TCIMACRO{\dint \limits_{\Omega}}%
%BeginExpansion
{\displaystyle\int\limits_{\Omega}}
%EndExpansion
\Gamma_{\frac{1}{2}}\left(  y-x\right)  f(y)d\Omega_{y}\right) \\
& =\left[  u\right]  _{g}\uplus\left[  u\right]  _{f}%
\end{align*}
with%

\[
\left[  u\right]  _{g}=%
%TCIMACRO{\dint \limits_{\partial\Omega}}%
%BeginExpansion
{\displaystyle\int\limits_{\partial\Omega}}
%EndExpansion
\Gamma_{\frac{1}{2}}\left(  y-x\right)  \upsilon_{\frac{1}{2}}\left(
y\right)  g(y)d\partial\Omega_{y}\text{ \ \ and \ \ }\left[  u\right]  _{f}=-%
%TCIMACRO{\dint \limits_{\Omega}}%
%BeginExpansion
{\displaystyle\int\limits_{\Omega}}
%EndExpansion
\Gamma_{\frac{1}{2}}\left(  y-x\right)  f(y)d\Omega_{y}\text{.}%
\]

\ 

$\left(  i\right)  $ Then since $W^{\frac{1}{2},2}\left(  \Omega\right)  $ is
an inner product space of an orthogonal decomposition we have%

\[
\left\langle \left[  u\right]  _{g},\left[  u\right]  _{f}\right\rangle
=\left\langle
%TCIMACRO{\dint \limits_{\partial\Omega}}%
%BeginExpansion
{\displaystyle\int\limits_{\partial\Omega}}
%EndExpansion
\Gamma_{\frac{1}{2}}\left(  y-x\right)  \upsilon_{\frac{1}{2}}\left(
y\right)  g(y)d\partial\Omega_{y},-%
%TCIMACRO{\dint \limits_{\Omega}}%
%BeginExpansion
{\displaystyle\int\limits_{\Omega}}
%EndExpansion
\Gamma_{\frac{1}{2}}\left(  y-x\right)  f(y)d\Omega_{y}\right\rangle
_{W^{\frac{1}{2},2}\left(  \Omega\right)  }=0\text{.}%
\]

\ 

$\left(  ii\right)  $ \ From the fact that the sum is an orthogonal sum, the
components obey the parallelogram law

\
\begin{align*}
& \parallel u\parallel_{W^{\frac{1}{2},2}\left(  \Omega\right)  }%
^{2}=\parallel%
%TCIMACRO{\dint \limits_{\partial\Omega}}%
%BeginExpansion
{\displaystyle\int\limits_{\partial\Omega}}
%EndExpansion
\Gamma_{\frac{1}{2}}\left(  y-x\right)  \upsilon_{\frac{1}{2}}\left(
y\right)  g(y)d\partial\Omega_{y}\parallel_{W^{\frac{1}{2},2}\left(
\Omega\right)  }^{2}+\parallel%
%TCIMACRO{\dint \limits_{\Omega}}%
%BeginExpansion
{\displaystyle\int\limits_{\Omega}}
%EndExpansion
\Gamma_{\frac{1}{2}}\left(  y-x\right)  f(y)d\Omega_{y}\parallel_{W^{\frac
{1}{2},2}\left(  \Omega\right)  }^{2}\\
& =\parallel\left[  u\right]  _{g}\parallel_{W^{\frac{1}{2},2}\left(
\Omega\right)  }^{2}+\parallel\left[  u\right]  _{f}\parallel_{W^{\frac{1}%
{2},2}\left(  \Omega\right)  }^{2}%
\end{align*}

$\left(  iii\right)  $ \ The third result follows from the fact that
\begin{align*}
\left\langle u,\text{ }\left[  u\right]  _{f}\right\rangle _{W^{\frac{1}{2}%
,2}\left(  \Omega\right)  }  & =\left\langle
%TCIMACRO{\dint \limits_{\partial\Omega}}%
%BeginExpansion
{\displaystyle\int\limits_{\partial\Omega}}
%EndExpansion
\Gamma_{\frac{1}{2}}\left(  y-x\right)  \upsilon_{\frac{1}{2}}\left(
y\right)  g(y)d\partial\Omega_{y}+\left(  -%
%TCIMACRO{\dint \limits_{\Omega}}%
%BeginExpansion
{\displaystyle\int\limits_{\Omega}}
%EndExpansion
\Gamma_{\frac{1}{2}}\left(  y-x\right)  f(y)d\Omega_{y}\right)  ,-%
%TCIMACRO{\dint \limits_{\Omega}}%
%BeginExpansion
{\displaystyle\int\limits_{\Omega}}
%EndExpansion
\Gamma_{\frac{1}{2}}\left(  y-x\right)  f(y)d\Omega_{y}\right\rangle
_{W^{\frac{1}{2},2}\left(  \Omega\right)  }\\
& =\left\langle -%
%TCIMACRO{\dint \limits_{\Omega}}%
%BeginExpansion
{\displaystyle\int\limits_{\Omega}}
%EndExpansion
\Gamma_{\frac{1}{2}}\left(  y-x\right)  f(y)d\Omega_{y},\text{ }-%
%TCIMACRO{\dint \limits_{\Omega}}%
%BeginExpansion
{\displaystyle\int\limits_{\Omega}}
%EndExpansion
\Gamma_{\frac{1}{2}}\left(  y-x\right)  f(y)d\Omega_{y}\right\rangle
_{W^{\frac{1}{2},2}\left(  \Omega\right)  }\\
& =\parallel-%
%TCIMACRO{\dint \limits_{\Omega}}%
%BeginExpansion
{\displaystyle\int\limits_{\Omega}}
%EndExpansion
\Gamma_{\frac{1}{2}}\left(  y-x\right)  f(y)d\Omega_{y}\parallel_{W^{\frac
{1}{2},2}\left(  \Omega\right)  }^{2}%
\end{align*}
and likewise $\left(  iv\right)  $ follows from
\begin{align*}
\left\langle u,\text{ }\left[  u\right]  _{g}\right\rangle _{W^{\frac{1}{2}%
,2}\left(  \Omega\right)  }  & =\text{ }\left\langle
%TCIMACRO{\dint \limits_{\partial\Omega}}%
%BeginExpansion
{\displaystyle\int\limits_{\partial\Omega}}
%EndExpansion
\Gamma_{\frac{1}{2}}\left(  y-x\right)  \upsilon_{\frac{1}{2}}\left(
y\right)  g(y)d\partial\Omega_{y},\text{ }%
%TCIMACRO{\dint \limits_{\partial\Omega}}%
%BeginExpansion
{\displaystyle\int\limits_{\partial\Omega}}
%EndExpansion
\Gamma_{\frac{1}{2}}\left(  y-x\right)  \upsilon_{\frac{1}{2}}\left(
y\right)  g(y)d\partial\Omega_{y}\right\rangle _{W^{\frac{1}{2},2}\left(
\Omega\right)  }\\
& =\text{ }\parallel%
%TCIMACRO{\dint \limits_{\partial\Omega}}%
%BeginExpansion
{\displaystyle\int\limits_{\partial\Omega}}
%EndExpansion
\Gamma_{\frac{1}{2}}\left(  y-x\right)  \upsilon_{\frac{1}{2}}\left(
y\right)  g(y)d\partial\Omega_{y}\parallel_{W^{\frac{1}{2},2}\left(
\Omega\right)  }^{2}%
\end{align*}
.
\end{proof}

The next result is for second order BVP.

\begin{proposition}
Let $f\in%
%TCIMACRO{\tciLaplace}%
%BeginExpansion
\mathcal{L}%
%EndExpansion
^{2}\left(  \Omega\right)  $ and $g_{1}\in W^{\frac{3}{2},2}\left(
\partial\Omega\right)  $,$g_{2}\in W^{\frac{1}{2},2}\left(  \partial
\Omega\right)  $, then the second order BVP interns of $D^{\frac{1}{2}}$
\begin{equation}
\left\{
\begin{array}
[c]{cc}%
\left(  D^{\frac{1}{2}}\right)  ^{2}u=f & \text{in }\Omega\\
u=\left(  g_{1},\text{ }g_{2}\right)  & \text{on }\partial\Omega
\end{array}
\right. \label{secondbvp}%
\end{equation}
with
\[
g_{1}=\tau u_{\mid\partial\Omega}\text{ \ and \ }g_{2}=\tau D^{\frac{1}{2}%
}u_{\mid\partial\Omega}%
\]
has a solution $u\in W^{1,2}\left(  \Omega\right)  $ given by%
\[
u\left(  x\right)  =%
%TCIMACRO{\dint \limits_{\partial\Omega}}%
%BeginExpansion
{\displaystyle\int\limits_{\partial\Omega}}
%EndExpansion
\Gamma_{\frac{1}{2}}\left(  y-x\right)  \upsilon_{\frac{1}{2}}\left(
y\right)  g_{1}\left(  y\right)  d\partial\Omega_{y}-%
%TCIMACRO{\dint \limits_{\Omega}}%
%BeginExpansion
{\displaystyle\int\limits_{\Omega}}
%EndExpansion%
%TCIMACRO{\dint \limits_{\partial\Omega}}%
%BeginExpansion
{\displaystyle\int\limits_{\partial\Omega}}
%EndExpansion
\Gamma_{\frac{1}{2}}\left(  z-x\right)  \Gamma_{\frac{1}{2}}\left(
y-z\right)  \upsilon_{\frac{1}{2}}\left(  z\right)  g_{2}\left(  z\right)
d\partial\Omega_{z}d\Omega_{y}+%
%TCIMACRO{\dint \limits_{\Omega}}%
%BeginExpansion
{\displaystyle\int\limits_{\Omega}}
%EndExpansion%
%TCIMACRO{\dint \limits_{\Omega}}%
%BeginExpansion
{\displaystyle\int\limits_{\Omega}}
%EndExpansion
\Gamma_{\frac{1}{2}}\left(  z-x\right)  \Gamma_{\frac{1}{2}}\left(
y-z\right)  f\left(  z\right)  d\Omega_{z}d\Omega_{y}%
\]
with the following properties

$\left(  i\right)  $
\begin{align*}
u\left(  x\right)   & =\left(
%TCIMACRO{\dint \limits_{\partial\Omega}}%
%BeginExpansion
{\displaystyle\int\limits_{\partial\Omega}}
%EndExpansion
\Gamma_{\frac{1}{2}}\left(  y-x\right)  \upsilon_{\frac{1}{2}}\left(
y\right)  g_{1}\left(  y\right)  d\partial\Omega_{y}-%
%TCIMACRO{\dint \limits_{\Omega}}%
%BeginExpansion
{\displaystyle\int\limits_{\Omega}}
%EndExpansion%
%TCIMACRO{\dint \limits_{\partial\Omega}}%
%BeginExpansion
{\displaystyle\int\limits_{\partial\Omega}}
%EndExpansion
\Gamma_{\frac{1}{2}}\left(  z-x\right)  \Gamma_{\frac{1}{2}}\left(
y-z\right)  \upsilon_{\frac{1}{2}}\left(  z\right)  g_{2}\left(  z\right)
d\partial\Omega_{z}d\Omega_{y}\right)  \uplus\left(
%TCIMACRO{\dint \limits_{\Omega}}%
%BeginExpansion
{\displaystyle\int\limits_{\Omega}}
%EndExpansion%
%TCIMACRO{\dint \limits_{\Omega}}%
%BeginExpansion
{\displaystyle\int\limits_{\Omega}}
%EndExpansion
\Gamma_{\frac{1}{2}}\left(  z-x\right)  \Gamma_{\frac{1}{2}}\left(
y-z\right)  f\left(  z\right)  d\Omega_{z}d\Omega_{y}\right) \\
& =\left[  u\right]  _{g_{1},g_{2}}\uplus\left[  u\right]  _{f}.
\end{align*}
$\left(  ii\right)  $%
\begin{align*}
& \parallel u\parallel_{W^{1,2}\left(  \Omega\right)  }^{2}=\parallel\left(
%TCIMACRO{\dint \limits_{\partial\Omega}}%
%BeginExpansion
{\displaystyle\int\limits_{\partial\Omega}}
%EndExpansion
\Gamma_{\frac{1}{2}}\left(  y-x\right)  \upsilon_{\frac{1}{2}}\left(
y\right)  g_{1}\left(  y\right)  d\partial\Omega_{y}-%
%TCIMACRO{\dint \limits_{\Omega}}%
%BeginExpansion
{\displaystyle\int\limits_{\Omega}}
%EndExpansion%
%TCIMACRO{\dint \limits_{\partial\Omega}}%
%BeginExpansion
{\displaystyle\int\limits_{\partial\Omega}}
%EndExpansion
\Gamma_{\frac{1}{2}}\left(  z-x\right)  \Gamma_{\frac{1}{2}}\left(
y-z\right)  \upsilon_{\frac{1}{2}}\left(  z\right)  g_{2}\left(  z\right)
d\partial\Omega_{z}d\Omega_{y}\right)  \parallel_{W^{1,2}\left(
\Omega\right)  }^{2}\\
+  & \parallel%
%TCIMACRO{\dint \limits_{\Omega}}%
%BeginExpansion
{\displaystyle\int\limits_{\Omega}}
%EndExpansion%
%TCIMACRO{\dint \limits_{\Omega}}%
%BeginExpansion
{\displaystyle\int\limits_{\Omega}}
%EndExpansion
\Gamma_{\frac{1}{2}}\left(  z-x\right)  \Gamma_{\frac{1}{2}}\left(
y-z\right)  f\left(  z\right)  d\Omega_{z}d\Omega_{y}\parallel_{W^{1,2}\left(
\Omega\right)  }^{2}\\
& =\parallel\left[  u\right]  _{g_{1},g_{2}}\parallel_{W^{1,2}\left(
\Omega\right)  }^{2}+\parallel\left[  u\right]  _{f}\parallel_{W^{1,2}\left(
\Omega\right)  }^{2}.
\end{align*}

\end{proposition}

\begin{proof}
Clearly since the input function $f\in%
%TCIMACRO{\tciLaplace}%
%BeginExpansion
\mathcal{L}%
%EndExpansion
^{2}\left(  \Omega\right)  $, which is the weekly second order derivative of
the solution $u$, we have $u$ to be in $W^{1,2}\left(  \Omega\right)  $. This
is because the Dirac operator $D$ is a regularity exponent diminishing
operator, between Sobolev spaces.

The repeated application of the integral representation given in (
\ref{borel-pompeiu}\ ) will be used. Let $v\left(  x\right)  =D^{\frac{1}{2}%
}u(x)$, then we have a first order BVP%

\[
\left\{
\begin{array}
[c]{cc}%
-D^{\frac{1}{2}}v=f & \text{in }\Omega\\
v=g_{2} & \text{on }\partial\Omega
\end{array}
\right.
\]
whose solution is given by
\[
v(x)=%
%TCIMACRO{\dint \limits_{\partial\Omega}}%
%BeginExpansion
{\displaystyle\int\limits_{\partial\Omega}}
%EndExpansion
\Gamma_{\frac{1}{2}}\left(  y-x\right)  \upsilon_{\frac{1}{2}}\left(
y\right)  g_{2}(y)d\partial\Omega_{y}+\left(  -%
%TCIMACRO{\dint \limits_{\Omega}}%
%BeginExpansion
{\displaystyle\int\limits_{\Omega}}
%EndExpansion
\Gamma_{\frac{1}{2}}\left(  y-x\right)  f(y)d\Omega_{y}\right)  \text{.}%
\]

But $D^{\frac{1}{2}}u=v$ and hence we have again a first order BVP%

\[
\left\{
\begin{array}
[c]{cc}%
-D^{\frac{1}{2}}u=f & \text{in }\Omega\\
u=g_{1} & \text{on }\partial\Omega
\end{array}
\right.
\]
with a solution%

\begin{align*}
u\left(  x\right)   & =%
%TCIMACRO{\dint \limits_{\partial\Omega}}%
%BeginExpansion
{\displaystyle\int\limits_{\partial\Omega}}
%EndExpansion
\Gamma_{\frac{1}{2}}\left(  y-x\right)  \upsilon_{\frac{1}{2}}\left(
y\right)  u\left(  y\right)  d\partial\Omega_{y}-%
%TCIMACRO{\dint \limits_{\Omega}}%
%BeginExpansion
{\displaystyle\int\limits_{\Omega}}
%EndExpansion
\Gamma_{\frac{1}{2}}\left(  y-x\right)  D^{\frac{1}{2}}u(y)d\Omega_{y}\\
& =%
%TCIMACRO{\dint \limits_{\partial\Omega}}%
%BeginExpansion
{\displaystyle\int\limits_{\partial\Omega}}
%EndExpansion
\Gamma_{\frac{1}{2}}\left(  y-x\right)  \upsilon_{\frac{1}{2}}\left(
y\right)  g_{1}\left(  y\right)  d\partial\Omega_{y}-%
%TCIMACRO{\dint \limits_{\Omega}}%
%BeginExpansion
{\displaystyle\int\limits_{\Omega}}
%EndExpansion
\Gamma_{\frac{1}{2}}\left(  y-x\right)  \left(
\begin{array}
[c]{c}%
%TCIMACRO{\dint \limits_{\partial\Omega}}%
%BeginExpansion
{\displaystyle\int\limits_{\partial\Omega}}
%EndExpansion
\Gamma_{\frac{1}{2}}\left(  z-y\right)  \upsilon_{\frac{1}{2}}\left(
z\right)  g_{2}(z)d\partial\Omega_{z}\\
+\left(  -%
%TCIMACRO{\dint \limits_{\Omega}}%
%BeginExpansion
{\displaystyle\int\limits_{\Omega}}
%EndExpansion
\Gamma_{\frac{1}{2}}\left(  z-y\right)  f(z)d\Omega_{z}\right)
\end{array}
\right)  d\Omega_{y}\\
& =%
%TCIMACRO{\dint \limits_{\partial\Omega}}%
%BeginExpansion
{\displaystyle\int\limits_{\partial\Omega}}
%EndExpansion
\Gamma_{\frac{1}{2}}\left(  y-x\right)  \upsilon_{\frac{1}{2}}\left(
y\right)  g_{1}\left(  y\right)  d\partial\Omega_{y}-%
%TCIMACRO{\dint \limits_{\Omega}}%
%BeginExpansion
{\displaystyle\int\limits_{\Omega}}
%EndExpansion%
%TCIMACRO{\dint \limits_{\partial\Omega}}%
%BeginExpansion
{\displaystyle\int\limits_{\partial\Omega}}
%EndExpansion
\Gamma_{\frac{1}{2}}\left(  y-x\right)  \Gamma_{\frac{1}{2}}\left(
z-y\right)  \upsilon_{\frac{1}{2}}\left(  z\right)  g_{2}(z)d\partial
\Omega_{z}d\Omega_{y}\\
& +%
%TCIMACRO{\dint \limits_{\Omega}}%
%BeginExpansion
{\displaystyle\int\limits_{\Omega}}
%EndExpansion%
%TCIMACRO{\dint \limits_{\Omega}}%
%BeginExpansion
{\displaystyle\int\limits_{\Omega}}
%EndExpansion
\Gamma_{\frac{1}{2}}\left(  y-x\right)  \Gamma_{\frac{1}{2}}\left(
z-y\right)  f(z)d\Omega_{z}d\Omega_{y}%
\end{align*}

\ \ 

We need to show that this sum is again an orthogonal sum from the orthogonal decomposition%

\[
W^{1,2}\left(  \Omega\right)  =A^{2,2}\left(  \Omega\right)  \oplus\left(
D^{\frac{1}{2}}\right)  ^{2}\left(  W_{0}^{2,2}\left(  \Omega\right)  \right)
\]
proven in \cite{dlakew1}. Clearly \ \ %

\[%
%TCIMACRO{\dint \limits_{\partial\Omega}}%
%BeginExpansion
{\displaystyle\int\limits_{\partial\Omega}}
%EndExpansion
\Gamma_{\frac{1}{2}}\left(  y-x\right)  \upsilon_{\frac{1}{2}}\left(
y\right)  g_{1}\left(  y\right)  d\partial\Omega_{y}-%
%TCIMACRO{\dint \limits_{\Omega}}%
%BeginExpansion
{\displaystyle\int\limits_{\Omega}}
%EndExpansion%
%TCIMACRO{\dint \limits_{\partial\Omega}}%
%BeginExpansion
{\displaystyle\int\limits_{\partial\Omega}}
%EndExpansion
\Gamma_{\frac{1}{2}}\left(  y-x\right)  \Gamma_{\frac{1}{2}}\left(
z-y\right)  \upsilon_{\frac{1}{2}}\left(  z\right)  g_{2}(z)d\partial
\Omega_{z}d\Omega_{y}%
\]
is annihilated by $\left(  D^{\frac{1}{2}}\right)  ^{2}$ since \
\begin{align*}
& \left(  D^{\frac{1}{2}}\right)  ^{2}\left(
%TCIMACRO{\dint \limits_{\partial\Omega}}%
%BeginExpansion
{\displaystyle\int\limits_{\partial\Omega}}
%EndExpansion
\Gamma_{\frac{1}{2}}\left(  y-x\right)  \upsilon_{\frac{1}{2}}\left(
y\right)  g_{1}\left(  y\right)  d\partial\Omega_{y}-%
%TCIMACRO{\dint \limits_{\Omega}}%
%BeginExpansion
{\displaystyle\int\limits_{\Omega}}
%EndExpansion%
%TCIMACRO{\dint \limits_{\partial\Omega}}%
%BeginExpansion
{\displaystyle\int\limits_{\partial\Omega}}
%EndExpansion
\Gamma_{\frac{1}{2}}\left(  y-x\right)  \Gamma_{\frac{1}{2}}\left(
z-y\right)  \upsilon_{\frac{1}{2}}\left(  z\right)  g_{2}(z)d\partial
\Omega_{z}d\Omega_{y}\right) \\
& =\left(  D^{\frac{1}{2}}\right)  ^{2}\left(
%TCIMACRO{\dint \limits_{\partial\Omega}}%
%BeginExpansion
{\displaystyle\int\limits_{\partial\Omega}}
%EndExpansion
\Gamma_{\frac{1}{2}}\left(  y-x\right)  \upsilon_{\frac{1}{2}}\left(
y\right)  g_{1}\left(  y\right)  d\partial\Omega_{y}\right)  -\left(
D^{\frac{1}{2}}\right)  ^{2}\left(
%TCIMACRO{\dint \limits_{\Omega}}%
%BeginExpansion
{\displaystyle\int\limits_{\Omega}}
%EndExpansion%
%TCIMACRO{\dint \limits_{\partial\Omega}}%
%BeginExpansion
{\displaystyle\int\limits_{\partial\Omega}}
%EndExpansion
\Gamma_{\frac{1}{2}}\left(  y-x\right)  \Gamma_{\frac{1}{2}}\left(
z-y\right)  \upsilon_{\frac{1}{2}}\left(  z\right)  g_{2}(z)d\partial
\Omega_{z}d\Omega_{y}\right) \\
& =-D^{\frac{1}{2}}\left(
%TCIMACRO{\dint \limits_{\partial\Omega}}%
%BeginExpansion
{\displaystyle\int\limits_{\partial\Omega}}
%EndExpansion
\Gamma_{\frac{1}{2}}\left(  z-y\right)  \upsilon_{\frac{1}{2}}\left(
z\right)  g_{2}(z)d\partial\Omega_{z}d\Omega_{y}\right)  =0\text{.}%
\end{align*}

Thus%

\[%
%TCIMACRO{\dint \limits_{\partial\Omega}}%
%BeginExpansion
{\displaystyle\int\limits_{\partial\Omega}}
%EndExpansion
\Gamma_{\frac{1}{2}}\left(  y-x\right)  \upsilon_{\frac{1}{2}}\left(
y\right)  g_{1}\left(  y\right)  d\partial\Omega_{y}-%
%TCIMACRO{\dint \limits_{\Omega}}%
%BeginExpansion
{\displaystyle\int\limits_{\Omega}}
%EndExpansion%
%TCIMACRO{\dint \limits_{\partial\Omega}}%
%BeginExpansion
{\displaystyle\int\limits_{\partial\Omega}}
%EndExpansion
\Gamma_{\frac{1}{2}}\left(  y-x\right)  \Gamma_{\frac{1}{2}}\left(
z-y\right)  \upsilon_{\frac{1}{2}}\left(  z\right)  g_{2}(z)d\partial
\Omega_{z}d\Omega_{y}\in A^{2,2}\left(  \Omega\right)  \text{.}%
\]

\ \ 

Next, we need to show that $\exists\xi\in W_{0}^{2,2}\left(  \Omega\right)  $
such that
\[%
%TCIMACRO{\dint \limits_{\Omega}}%
%BeginExpansion
{\displaystyle\int\limits_{\Omega}}
%EndExpansion%
%TCIMACRO{\dint \limits_{\Omega}}%
%BeginExpansion
{\displaystyle\int\limits_{\Omega}}
%EndExpansion
\Gamma_{\frac{1}{2}}\left(  y-x\right)  \Gamma_{\frac{1}{2}}\left(
z-y\right)  f(z)d\Omega_{z}d\Omega_{y}=\left(  D^{\frac{1}{2}}\right)  ^{2}%
\xi\left(  x\right)  \text{.}%
\]
\ 

Clearly
\[
\xi\left(  x\right)  _{\mid\partial\Omega}=\left(  \xi\left(  x\right)
_{\mid\partial\Omega},D^{\frac{1}{2}}\xi\left(  x\right)  _{\mid\partial
\Omega}\right)  =\left(  0,0\right)  \text{.}%
\]
By applying the result of integration by parts above twice and the fact that%
\[
\xi_{\mid\partial\Omega}=0
\]
we have \ \
\[
\xi\left(  x\right)  =%
%TCIMACRO{\dint \limits_{\Omega}}%
%BeginExpansion
{\displaystyle\int\limits_{\Omega}}
%EndExpansion%
%TCIMACRO{\dint \limits_{\Omega}}%
%BeginExpansion
{\displaystyle\int\limits_{\Omega}}
%EndExpansion%
%TCIMACRO{\dint \limits_{\Omega}}%
%BeginExpansion
{\displaystyle\int\limits_{\Omega}}
%EndExpansion%
%TCIMACRO{\dint \limits_{\Omega}}%
%BeginExpansion
{\displaystyle\int\limits_{\Omega}}
%EndExpansion
\Gamma_{\frac{1}{2}}\left(  y-x\right)  \Gamma_{\frac{1}{2}}\left(
z-y\right)  \Gamma_{\frac{1}{2}}\left(  w-z\right)  \Gamma_{\frac{1}{2}%
}\left(  q-w\right)  f\left(  q\right)  d\Omega_{q}d\Omega_{w}d\Omega
_{z}d\Omega_{y}\in W_{0}^{2,2}\left(  \Omega\right)  \text{.}%
\]

Setting \ %

\[
\left[  u\right]  _{g_{1},g_{2}}=%
%TCIMACRO{\dint \limits_{\partial\Omega}}%
%BeginExpansion
{\displaystyle\int\limits_{\partial\Omega}}
%EndExpansion
\Gamma_{\frac{1}{2}}\left(  y-x\right)  \upsilon_{\frac{1}{2}}\left(
y\right)  g_{1}\left(  y\right)  d\partial\Omega_{y}-%
%TCIMACRO{\dint \limits_{\Omega}}%
%BeginExpansion
{\displaystyle\int\limits_{\Omega}}
%EndExpansion%
%TCIMACRO{\dint \limits_{\partial\Omega}}%
%BeginExpansion
{\displaystyle\int\limits_{\partial\Omega}}
%EndExpansion
\Gamma_{\frac{1}{2}}\left(  y-x\right)  \Gamma_{\frac{1}{2}}\left(
z-y\right)  \upsilon_{\frac{1}{2}}\left(  z\right)  g_{2}(z)d\partial
\Omega_{z}d\Omega_{y}%
\]
and \ \ %

\[
\left[  u\right]  _{f}=%
%TCIMACRO{\dint \limits_{\Omega}}%
%BeginExpansion
{\displaystyle\int\limits_{\Omega}}
%EndExpansion%
%TCIMACRO{\dint \limits_{\Omega}}%
%BeginExpansion
{\displaystyle\int\limits_{\Omega}}
%EndExpansion
\Gamma_{\frac{1}{2}}\left(  y-x\right)  \Gamma_{\frac{1}{2}}\left(
z-y\right)  f(z)d\Omega_{z}d\Omega_{y}%
\]

we have
\[
u=\left[  u\right]  _{g_{1},g_{2}}\uplus\left[  u\right]  _{f}%
\]
which proves $\left(  i\right)  $.

\ 

$\left(  ii\right)  $ follows from the fact that
\[
u=\left[  u\right]  _{g_{1},g_{2}}\uplus\left[  u\right]  _{f}\text{.}%
\]

\ \ 
\end{proof}

\section{\textbf{Applications}}

We consider the application of fractional orthogonal decomposition. A common
field of investigation of a physical phenomena to consider is a vector field
$F$ of electromagnetism, which is the orthogonal sum of the fractional
gradient of the scalar potential $\psi:=-\left(  \nabla^{2}\right)
^{\frac{-3}{4}}\left(  \nabla.F\right)  $ and the fractional curl of the
vector potential $\phi:=\left(  \nabla^{2}\right)  ^{\frac{-3}{4}}\left(
\nabla\times F\right)  $. That is%
\[
F=\underset{\text{irrotational/normal }}{\underbrace{\nabla^{\frac{1}{2}}\psi
}}\uplus\underset{\text{solenoidal/tangential}}{\underbrace{Curl^{\frac{1}{2}%
}\phi}}%
\]

One can verify that the curl
\[
\nabla^{\frac{1}{2}}\times\left(  -\nabla^{\frac{1}{2}}\psi\right)  =0
\]
That is $\nabla^{\frac{1}{2}}\psi$ is irrotational and the divergence of the
fractional half vector potential
\[
\nabla^{\frac{1}{2}}\times\phi,\nabla^{\frac{1}{2}}\cdot\left(  \nabla
^{\frac{1}{2}}\times\phi\right)  =0.
\]

We need to show that the sum is an orthogonal in the sense of Hilbert space, $%
%TCIMACRO{\tciLaplace}%
%BeginExpansion
\mathcal{L}%
%EndExpansion
^{2}\left(  \Omega\right)  $ That is

$\qquad$%
\[
(i)\ \ \langle\nabla^{\frac{1}{2}}\psi,\nabla^{\frac{1}{2}}\times\phi\rangle=0
\]

\qquad%
\[
(ii)\ \parallel F\parallel^{2}=\parallel\nabla^{\frac{1}{2}}\psi\parallel
^{2}+\parallel\nabla^{\frac{1}{2}}\times\phi\parallel^{2}%
\]

\ 

We investigate:
\[
\langle-\nabla^{\frac{1}{2}}\psi,\nabla^{\frac{1}{2}}\times\phi\rangle
=\int_{\Omega}\left(  -\nabla^{\frac{1}{2}}\psi\right)  \left(  \nabla
^{\frac{1}{2}}\times\phi\right)  dv
\]

From
\[
\nabla^{\frac{1}{2}}\cdot\left(  \psi\nabla^{\frac{1}{2}}\times\phi\right)
=\psi\left(  \nabla^{\frac{1}{2}}\cdot\nabla^{\frac{1}{2}}\times\phi\right)
+\nabla^{\frac{1}{2}}\psi\cdot\nabla^{\frac{1}{2}}\times\phi
\]

we have%
\[
\nabla^{\frac{1}{2}}\psi\cdot\nabla^{\frac{1}{2}}\times\phi=\nabla^{\frac
{1}{2}}\cdot\left(  \psi\left(  \nabla^{\frac{1}{2}}\times\phi\right)
\right)  -\psi\left(  \underset{\underset{0}{\shortparallel}%
}{\underbrace{\nabla^{\frac{1}{2}}\cdot\left(  \nabla^{\frac{1}{2}}\times
\phi\right)  }}\right)
\]

Thus, $\ $%
\begin{align*}
\langle-\nabla^{\frac{1}{2}}\psi,\nabla^{\frac{1}{2}}\times\phi\rangle &
=-\int_{\Omega}\nabla^{\frac{1}{2}}\cdot\left(  \psi\left(  \nabla^{\frac
{1}{2}}\times\phi\right)  \right)  dv\\
& =-%
%TCIMACRO{\doint \limits_{\partial\Omega}}%
%BeginExpansion
{\displaystyle\oint\limits_{\partial\Omega}}
%EndExpansion
\psi\underset{\text{irrotational}}{\underbrace{\left(  \nabla^{\frac{1}{2}%
}\times\phi\right)  }}\cdot\overset{\wedge}{n}d\partial\Omega\\
& =0
\end{align*}

Since
\[
\underset{\text{irrotational/normal}}{\underbrace{\left(  \nabla^{\frac{1}{2}%
}\times\phi\right)  }}\cdot\underset{\text{normal vector}%
}{\underbrace{\overset{\wedge}{n}_{\mid\partial\Omega}}}=0
\]

\begin{proposition}
$\parallel F\parallel^{2}=\parallel F_{\text{tangential}}\parallel
^{2}+\parallel F_{\text{normal}}\parallel^{2}=\int_{\Omega}\mid\nabla
^{\frac{1}{2}}\psi\mid^{2}dv+\int_{\Omega}\mid\nabla^{\frac{1}{2}}\times
\phi\mid^{2}dv$
\end{proposition}

\end{document}